\documentclass{article}
\usepackage[hidelinks]{hyperref}
\usepackage{amssymb}
\usepackage[shortalphabetic]{amsrefs}
\usepackage{amsmath}
\usepackage{upgreek}
\usepackage{amsthm}
\usepackage{mathdots}
\usepackage{xypic}
\usepackage{url}
\usepackage{enumerate}
\usepackage{mathtools}
\usepackage{amsxtra}
\usepackage{xcolor}
 \usepackage[refpage]{nomencl}
 \usepackage{enumitem}

\newtheorem{thm}{Theorem}[section]
\newtheorem{lem}[thm]{Lemma}
\newtheorem{cor}[thm]{Corollary}
\newtheorem{prop}[thm]{Proposition}

\title{L-packets over strong real forms}

\author{N. Arancibia Robert, P. Mezo}

\begin{document}

\maketitle

\begin{abstract}
Langlands defined L-packets for real reductive groups. In order to refine the local Langlands correspondence,
Adams-Barbasch-Vogan combined L-packets over all real forms belonging to an
inner class 
(\cite{abv}). Using different methods, Kaletha also defines such combined L-packets with a
refinement to the local Langlands correspondence (\cite{kal}).  We prove that the
L-packets of Adams-Barbasch-Vogan and Kaletha are the same and are
parameterized identically.  
\end{abstract}

\section{Introduction}

Let $G$ be a connected reductive algebraic group defined over
$\mathbb{R}$. Langlands defined a partition of the (infinitesimal equivalence
classes of) irreducible admissible representations of $G(\mathbb{R})$
into L-packets (\cite{Langlands}).  The L-packets $\Pi_{\phi}$ are parameterized by
(conjugacy classes of) L-homomorphisms
\begin{equation}
  \label{Lhom}
  \phi: W_{\mathbb{R}} \rightarrow {^\vee}G^{\Gamma}
\end{equation}
from the Weil group $W_{\mathbb{R}}$ of $\Gamma =
\mathrm{Gal}(\mathbb{C}/\mathbb{R}$) to the L-group
${^\vee}G^{\Gamma}$ (\cite{borel}*{Section 8}).

Langlands proposed a further
parameterization of the representations in each L-packet in terms of
the centralizer ${^\vee}G_{\phi}$ of $\phi(W_{\mathbb{R}})$ in the
dual group
${^\vee}G$.  Specifically, the representations in $\Pi_{\phi}$ were to
correspond to characters of the (abelian) component group
${^\vee}G_{\phi}/(^{\vee}G_{\phi})^{0}$.  Shelstad completed this
proposal for a wide class of groups (\cite{She82}), but was obliged to
introduce formal ``ghost'' representations in order to make the
correspondence one-to-one.

Adams, Barbasch and Vogan proposed a different approach to achieve a
one-to-one correspondence (\cite{abv}*{Theorem 10.11}).  Rather than
defining a packet for a single real form $G(\mathbb{R})$, they defined a packet of
representations over several real forms appearing in the inner class of
$G(\mathbb{R})$.  We continue to call these compound packets
``L-packets'', but denote them by
\begin{equation}
  \label{abvpacket}
\Pi^{\mathrm{ABV}}_{\phi} = \left\{ \pi^{\mathrm{ABV}}_{\tau}: \tau  \in ({^\vee}G_{\phi}/({^\vee}G_{\phi})^{0})^{\wedge} 
\right\}.
\end{equation}
In this preliminary version of the L-packet, each character $\tau$
determines a real form of $G$ which is inner to a quasisplit form, and
$\pi_{\tau}^{\mathrm{ABV}}$ 
is (an infinitesimal equivalence class of) a representation of the
inner form. Unfortunately, not every real form of $G$ appears in such
an L-packet. (Only the \emph{pure real forms} appear (Section
\ref{strongrigid}).)  In order to include all of the real forms of $G$ one
must introduce algebraic coverings
$$1 \rightarrow \hat{J} \rightarrow {^\vee}G^{\hat{J}} \rightarrow
{^\vee}G \rightarrow 1.$$
in which we take $\hat{J}$ to be a finite abelian group.  The
more general L-packets, which encompass all real forms in an inner
class, take the form
$$\Pi^{\mathrm{ABV}}_{\phi,J} = \left\{ \pi^{\mathrm{ABV}}_{\tau}: \tau  \in ({^\vee}G_{\phi}^{\hat{J}}/({^\vee}G_{\phi}^{\hat{J}})^{0})^{\wedge} 
\right\}.$$
Here, $J$ is a (sufficiently large) finite subgroup of the centre $Z(G)$, defined over
$\mathbb{R}$, and is related to $\hat{J}$ through \cite{abv}*{Lemma
  10.2}.  The real forms parameterized by the characters $\tau$ are
called \emph{strong real forms of type $J$}.  

In \cite{vogan_local_langlands} Vogan presented ideas towards
defining compound L-packets for reductive groups over p-adic fields
as well.  A basic problem in this endeavour is to extend the notion of
strong real form to groups over local fields.  In characteristic zero,
this problem was solved by Kaletha,
who described the analogue as a \emph{rigid inner twist}
(\cite{kal}).  Kaletha combined his theory of rigid inner twists with
the work of Langlands and Shelstad on L-packets for real groups.  In
doing so, he defined compound L-packets for \emph{tempered}
L-homomorphisms (\cite{kal}*{Section 5.4}).  Let us assume for the
moment that $\phi$ is tempered, \emph{i.e.} has
bounded image in ${^\vee}G$.  The compound L-packets of
Kaletha, Langlands and Shelstad run over representations of strong
real forms and are in one-to-one correspondence with the characters of
${^\vee}G_{\phi}^{\hat{J}}/({^\vee}G_{\phi}^{\hat{J}})^{0}$.  We
denote these packets by
$$\Pi^{\mathrm{KLS}}_{\phi,J} = \left\{ \pi^{\mathrm{KLS}}_{\tau}: \tau  \in ({^\vee}G_{\phi}^{\hat{J}}/({^\vee}G_{\phi}^{\hat{J}})^{0})^{\wedge} 
\right\}.$$
Our main theorem is
\begin{thm}
  \label{mainthm}
Suppose $\phi$ is tempered. Then  for all $\tau \in
  ({^\vee}G_{\phi}^{\hat{J}}/({^\vee}G_{\phi}^{\hat{J}})^{0})^{\wedge}$
$$\pi^{\mathrm{ABV}}_{\tau} = \pi^{\mathrm{KLS}}_{\tau}.$$ 
\end{thm}
The equation in this theorem is to be interpreted as an equality of
infinitesimal equivalence classes of irreducible
representations. It should also be noted that  $\tau$
determines a strong real form using
the approach of Adams-Barbasch-Vogan, and determines another strong real form
using the approach of Kaletha.  Kaletha's approach relies heavily on
Galois cohomology, and this is completely absent in the approach of
Adams-Barbasch-Vogan.  It is implicit in the theorem that the
two strong real forms coincide. 

An essential feature in the two approaches is the choice of a maximal
torus related to $\phi$.  The choice of tori in the two approaches
differs significantly when the representations of the
L-packet have singular infinitesimal character.  In this case we
relate the two different tori through Cayley transforms. The
representations are defined through the two tori and we identify them
through Hecht-Schmid character identities (\cite{Speh-Vogan}*{Section
  5}).  

There are two corollaries to Theorem \ref{mainthm} to which
we wish to allude.  The
first pertains to  the assumption of $\phi$ being tempered.  The only
reason for this assumption is the absence of the definition of the
compound L-packets in \cite{kal}*{Section 5.6} for general $\phi$.
The extension of this definition to arbitrary $\phi$ may be completed by
imitating  Langlands' definition of arbitrary L-packets from
tempered ones (\cite{borel}*{Section 11.3}). One starts with the
compound essentially tempered L-packet of a Levi subgroup.  The general L-packet is then obtained by
taking  Langlands quotients of representations induced from the Levi
subgroup.  Although Adams-Barbasch-Vogan do not define their L-packets
using intermediate Levi subgroups, their definition is also 
consistent with this approach.  In this way, the extension of Theorem
\ref{mainthm} to arbitrary $\phi$ is an unremarkable exercise.

The second corollary pertains to the theory of endoscopy.  Both Kaletha and Adams-Barbasch-Vogan define a notion
of endoscopic datum for ${^\vee}G^{\hat{J}}$ which extends the usual
one for ${^\vee}G$ (\cite{kal}*{Section 5.3}, \cite{abv}*{Definition
  26.15}), \cite{Langlands-Shelstad}).  In
the absence of the technicality of $z$-extensions
(\cite{Langlands-Shelstad}*{Section 4.4}), both definitions are easily
seen to coincide, and to boil down to an element $s \in {^\vee}G^{\hat{J}}$ and
a connected reductive group $H$.  Let us restrict to this setting and
further assume that $\phi$ is a tempered L-parameter for $H$.
In this setting $\phi$ yields a compound L-packet for $H$ and also for
$G$.
Suffice it to say that both
Kaletha and Adams-Barbasch-Vogan prove character identities between
the resulting L-packets (\cite{kal}*{(5.11), Proposition 5.10},
\cite{abv}*{Proposition 26.7, Definition 26.18}).   The chief
coefficients in both character identities are of the from $\tau(s)$,
where $\tau$ is given in Theorem \ref{mainthm}.  From the
equivalence of the representations attached to $\tau$ it becomes apparent
that both endoscopic character identities are the same.  This fact
implies that the endoscopic lifting of tempered L-packets in the two
perspectives is the same as well. (See \cite{abv}*{\emph{p.} 289}, where
this is left as a ``straightforward exercise''.)

We anticipate that Theorem \ref{mainthm} will be valuable in similar
endoscopic identities for compound Arthur packets of classical and
unitary groups (\cite{Arthur}*{Chapter 9}, \cite{Mok}, \cite{Kaletha-Minguez},
\cite{MR3},  \cite{abv}*{Theorem 26.25}, \cite{aam}).

This work is organized as follows.  In Section \ref{strongrigid} we
review the equivalent definitions of strong real forms and rigid inner twists
over the real numbers.  We recall how representations are attached to these
objects, and how pairings are attached to these objects.  The
relationships between the attached representations and pairings are
discussed.  Section \ref{abvlp} introduces the L-packets of \cite{abv}
without any assumptions on $\phi$.  Section \ref{klslp} introduces the
L-packets of \cite{kal} for tempered $\phi$.   The comparison of the
two kinds of L-packets requires a comparison of certain component
groups and the pairings between them and strong real forms.  This is
done in Sections \ref{compcomp}-\ref{compsrf} and is the core of the
paper. 
The main point is that the component groups differ by Cayley
transforms and that representations in the packets are seen to be
equivalent through  Hecht-Schmid character identities.   This is
briefly explained in Section \ref{thmsec}.

\section{Strong real forms and rigid inner twists}
\label{strongrigid}

We first recall the notion of strong real form in \cite{abv}.  We then review how Kaletha's notion of rigid inner twists, which are valid for groups over local fields, are equivalent to strong real forms when  the local field is $\mathbb{R}$ (\cite{kal}*{Section 5.2}).

We mention in passing that there is yet another equivalent notion to
strong real forms, namely that of \emph{strong involutions}
(\cite{Adams-Fokko}*{Definition 5.5}, \cite{Adams-Taibi}*{Remark
  8.13}).  Although we do not pursue strong involutions here, they are
often, for good reasons,  preferred over strong real forms.  One could
reformulate our results purely in terms of strong involutions.

Let $G$ be a connected reductive complex algebraic group.  Let
$G^{\Gamma}$ be a group containing $G$ as an index two subgroup, with
the additional condition that every element in $G^{\Gamma} -G$ acts by
conjugation on $G$ as an antiholomorphic automorphism.  More
precisely, given any $\delta \in G^{\Gamma} - G$ and an algebraic function $f \in \mathbb{C}[G]$, the function
$$g \mapsto \overline{f(\mathrm{Int}(\delta)(g))}, \quad g \in G$$
is also in $\mathbb{C}[G]$.  The group $G^{\Gamma}$ is a \emph{weak extended group containing} $G$ (\cite{abv}*{Definition 2.13}).  A \emph{strong real form of} $G^{\Gamma}$ is an element $\delta \in G^{\Gamma} -G$ such that $\delta^{2} \in G$ belongs to the centre $Z(G)$ and has finite order.  The strong real form $\delta$ determines a real form $\mathrm{Int}(\delta)$ and a corresponding group of real points
$$G(\mathbb{R}, \delta) = \{ g \in G : \mathrm{Int}(\delta)(g) = g \}.$$
Two strong real forms of $G^{\Gamma}$ are \emph{equivalent} if they are conjugate under $G$.

These definitions are enriched by the addition of a $G$-conjugacy
class $\mathcal{W}$ of a triple
\begin{equation}
  \label{whittdatum}
  (\delta_{q}, N, \chi)
\end{equation}
in which $\delta_{q}$ is a strong real form of $G^{\Gamma}$, $N \subset G$ is a maximal unipotent subgroup normalized by $\delta_{q}$, and $\chi$ is a unitary character of $N(\mathbb{R}, \delta_{q})$ which is non-trivial on each simple restricted root subgroup.  We fix such a $\mathcal{W}$.  The pair $(G^{\Gamma}, \mathcal{W})$ is called an \emph{extended group for}  $G$ (\cite{abv}*{Definition 1.12}).  In essence, $\mathcal{W}$ specifies a Whittaker datum for the quasisplit real form $G(\mathbb{R}, \delta_{q})$.

The extended groups for $G$ are classified in \cite{abv}*{Proposition 3.6}.  There it is shown that two extended groups $(G^{\Gamma}, \mathcal{W})$ and $((G^{\Gamma})', \mathcal{W}')$ for $G$ are equivalent only if $\delta_{q}^{2} = (\delta_{q}')^{2}$ for any choices $(\delta_{q},N,\chi) \in \mathcal{W}$ and $(\delta_{q}',N, \chi) \in \mathcal{W}'$.  This allows for inequivalent extended groups when the centre of $G$ is non-trivial.  Nevertheless, as noted on \cite{abv}*{\emph{p.} 46}, there appears to be no reason to prefer one extended group over another.  Thus, from now on, we assume that $\delta_{q}^{2} = 1 \in Z(G)$.  We fix $(\delta_{q},N,\chi) \in \mathcal{W}$ so that as an internal semidirect product
\begin{equation}
  \label{extG}
  G^{\Gamma} = G \rtimes \langle \delta_{q} \rangle.
\end{equation}
This also fixes the unique Borel subgroup $B \supset N$.  Let
$Z(G)^{\mathrm{tor}}$ be the torsion subgroup of $Z(G)$, \emph{i.e.}
the elements of finite order.  Let $J$ be a subgroup of $ Z(G)^{\mathrm{tor}}$.  A strong real form $\delta \in G^{\Gamma} -G$ is \emph{of type J} if $\delta^{2} \in J$ (\cite{abv}*{Definition 10.10}). We say that it is a \emph{pure real form} if it is of type $\{1\}$, \emph{i.e.}  $\delta^{2} = 1$.  Obviously, $\delta_{q}$ is a pure real form.

Kaletha defines analogues of strong real forms uniformly for groups
over any local field of characteristic zero (\cite{kal}*{Section
  5.1}).  These analogues are called rigid inner twists.  We summarize
and paraphrase his construction of rigid inner twists over the real
local field.   A \emph{rigid inner twist of} $(G, \delta_{q})$ is a
pair $(\psi, z)$.  The first element $\psi: G \rightarrow G'$ is a
$\mathbb{C}$-isomorphism of algebraic groups in which the group $G$
has the quasisplit $\mathbb{R}$-structure defined by
$\mathrm{Int}(\delta_{q})$. The group $G'$ is taken to have an
$\mathbb{R}$-structure defined by $\sigma'$ such that  $\psi^{-1}
\circ \sigma' \circ \psi \circ \mathrm{Int}(\delta_{q}) =
\mathrm{Int}(g_{\psi})$ for some $g_{\psi} \in G$.  The second element
defining the rigid inner twist is a 1-cocycle $z \in Z^{1}(W,G)$.
Here, $W$ is an extension of $\Gamma = \mathrm{Gal}(\mathbb{C}/\mathbb{R})$
$$1 \rightarrow u \rightarrow W \rightarrow \Gamma \rightarrow 1$$
in which $u$ is a pro-algebraic group (\cite{kal}*{Sections 3.1-3.2}).  The 1-cocycle is defined to satisfy two conditions.  The first is that the restriction $z_{|u}$ is an algebraic homomorphism into a finite subgroup $J$ of $Z(G)$, defined over $\mathbb{R}$.  The second is that $\mathrm{Int}(z(\sigma)) = \mathrm{Int}( g_{\psi})$, that is
\begin{equation}
  \label{zisg}
\psi^{-1} \circ \sigma' \circ \psi \circ \mathrm{Int}(\delta_{q}) = \mathrm{Int}(z(\sigma)),
\end{equation}
where $\sigma$ is a canonical element in $W$ which maps to the non-trivial element in $\Gamma$.  In this case, the rigid inner twist $(\psi,z)$ is said to be \emph{realized by} $J$.  Observe that the finiteness condition on the subgroup $J \subset Z(G)$ is stronger than the condition on $J$ for strong real forms stated above. A rigid inner twist is
\emph{pure} if it is realized by $\{1\}$.

Let $(\psi,z)$ and $(\psi_{1}, z_{1})$ be rigid inner twists of $(G, \delta_{q})$ with $\psi: G \rightarrow G'$ and $\psi_{1}: G \rightarrow G_{1}'$.  An \emph{isomorphism} between them is a pair $(f,g)$ in which $f:G' \rightarrow G_{1}'$ is an $\mathbb{R}$-isomorphism and $g \in G$.  This pair is required to satisfy: 
\begin{enumerate}
\item $\psi_{1} \circ \mathrm{Int}(g) = f \circ \psi$
 
\item $z_{1}(w) = g z(w) \, w\cdot(g^{-1})$ for all  $w \in W$ ($W$ acts on $G$ through $\Gamma$ alone.)
\end{enumerate}
The first property in the definition is equivalent to the commutativity of the diagram
$$\xymatrix@1{G \ar[r]^{\psi} \ar[d]_{\mathrm{Int}(g)} & G' \ar[d]^{f} \\
  G \ar[r]_{\psi_{1}}& G_{1}'}$$
Let us make the $\mathbb{R}$-structures of the groups in this diagram more explicit by writing
$$\xymatrix@1{(G,\mathrm{Int}(\delta_{q})) \ar[r]^{\psi} \ar[d]_{\mathrm{Int}(g)} & (G',\sigma') \ar[d]^{f} \\
  (G,\mathrm{Int}(\delta_{q}))  \ar[r]_{\psi_{1}}& (G_{1}',\sigma_{1}')}$$
Here, the $\Gamma$-action of the group in each pair is given by the automorphism in the second entry, and $f \circ \sigma' \circ f^{-1} = \sigma_{1}'$. 
\begin{lem}
  \label{idiso}
Let $\mathrm{id}:G \rightarrow G$ be the identity map and $\sigma \in
W$ be the canonical element mapping to the non-trivial element in
$\Gamma$ (\ref{zisg}).  Suppose $(\psi,z)$ is a rigid inner twist of
$(G,\delta_{q})$.  Then $(\psi^{-1},1)$ is an isomorphism between
$(\psi,z)$ and $(\mathrm{id},z)$, where the $\mathbb{R}$-structure of
$G$ in the codomain of $\mathrm{id}$ is defined by $\mathrm{Int}(z(\sigma) \delta_{q})$.
\end{lem}
\begin{proof} The second property in the definition of isomorphism of
  rigid inner twists is trivially satisfied.   We must therefore show
  the first property of the definition for $f = \psi^{-1}$,
  \emph{i.e.} show that  
$$\xymatrix@1{(G,\mathrm{Int}(\delta_{q})) \ar[r]^{\psi} \ar[d]_{\mathrm{Int}(1)} & (G',\sigma') \ar[d]^{\psi^{-1}} \\
  (G,\mathrm{Int}(\delta_{q}))  \ar[r]_{\mathrm{id}}&
    (G,\mathrm{Int}(z(\sigma)\delta_{q}) )}$$
  is a commutative diagram.  
The commutativity follows from (\ref{zisg}).
\end{proof}

Suppose $(\psi, z)$ is a rigid inner twist of $(G,\delta_{q})$.  According to \cite{kal}*{Theorem 5.2} the map
\begin{equation}
  \label{5.2map}
  (\psi,z) \mapsto z(\sigma)\delta_{q}
\end{equation}
takes values in the set of strong real forms of $G^{\Gamma}$.  In
addition, the map passes to a bijection from isomorphism classes of
rigid inner twists to equivalence classes of strong real forms.  It
carries  rigid inner twists realized by $J \subset Z(G)$ to strong real forms of type $J$.   In particular, it carries pure rigid inner twists to pure real forms.

\subsection{Representations}

A \emph{representation of a strong real form of} $G^{\Gamma}$ is a pair $(\pi, \delta)$ in which $\delta$ is a strong real form of $G^{\Gamma}$ and $\pi$ is an admissible representation of $G(\mathbb{R}, \delta)$.  Two such representations $(\pi, \delta)$ and $(\pi_{1}, \delta_{1})$ are \emph{equivalent} if there exists $g \in G$ such that $\delta_{1} = g\delta g^{-1}$ and $\pi_{1}$ is infinitesimally equivalent to $\pi \circ \mathrm{Int}(g^{-1})$ (\cite{abv}*{Definition 2.13}).

In the same vein, a \emph{representation of a rigid inner twist of}
$(G,\delta_{q})$ is a triple $(\psi, z, \pi)$ in which $\psi:G
\rightarrow G'$ is the $\mathbb{C}$-isomorphism of a rigid inner twist
$(\psi,z)$ and $\pi$ is an admissible representation of
$G'(\mathbb{R})$.  Suppose $(\psi_{1}, z_{1}, \pi_{1})$ is another
representation of a rigid inner twist of $(G, \delta_{q})$ with
$\psi_{1} : G \rightarrow G_{1}'$.  An \emph{isomorphism} between $(\psi,z,\pi)$ and $(\psi_{1}, z_{1}, \pi_{1})$ is an isomorphism $(f,g)$ between $(\psi,z)$ and $(\psi_{1},z_{1})$ such that 
$\pi_{1}$ is infinitesimally equivalent to $\pi \circ f^{-1}$.

Lemma \ref{idiso} implies that any representation $(\psi,z, \pi)$ of a rigid inner twist of $(G,\delta_{q})$ is isomorphic to $(\mathrm{id}, z, \pi \circ \psi)$.  Furthermore, the map (\ref{5.2map}) makes it  clear that $(z(\sigma)\delta_{q}, \pi \circ \psi)$ is a representation of a strong real from of $G^{\Gamma}$.  This defines a map
\begin{equation}
  \label{repmap}
  (\psi,z, \pi) \mapsto (z(\sigma)\delta_{q}, \pi \circ \psi)
\end{equation}
from representations of rigid inner twists of $(G,\delta_{q})$ to representations of strong real forms of $G^{\Gamma}$.
\begin{prop}
The map (\ref{repmap}) passes to a bijection of isomorphism classes of representations of rigid inner twists of $(G, \delta_{q})$ to equivalence classes of representations of strong real forms of $G^{\Gamma}$. 
\end{prop}
\begin{proof}
  Suppose $(f,g)$ is a isomorphism between  $(\psi,z, \pi)$ and $(\psi_{1}, z_{1}, \pi_{1})$.  We must prove that $(z(\sigma)\delta_{q}, \pi \circ \psi)$ is equivalent to $(z_{1}(\sigma)\delta_{q}, \pi_{1} \circ \psi_{1})$.  By definition
  $$\pi \circ \psi = \pi_{1} \circ f \circ \psi = \pi_{1} \circ \psi_{1} \circ \mathrm{Int}(g)$$
  and so
  \begin{align*}
    (z_{1}(\sigma)\delta_{q}, \pi_{1} \circ \psi_{1}) & = (z_{1}(\sigma)\delta_{q}, \pi \circ \psi \circ \mathrm{Int}(g^{-1}))\\
    &\sim (g^{-1} z_{1}(\sigma)  \delta_{q} g, \pi \circ \psi)\\
    & = (g^{-1}z_{1}(\sigma) \, \mathrm{Int}(\delta_{q})( g) \, \delta_{q}, \pi \circ \psi).
\end{align*}
  The action of $\sigma \in W$ on $G$ is by
  $\mathrm{Int}(\delta_{q})$.  The second property in the definition
  of the isomorphism $(f,g)$ therefore implies
  $$ (g^{-1}z_{1}(\sigma) \, \mathrm{Int}(\delta_{q})( g) \delta_{q}, \pi \circ \psi) = (z(\sigma) \delta_{q}, \pi \circ \psi),$$
  completing the proof of the equivalence. This proves that (\ref{repmap}) passes to a map from isomorphism classes to equivalence classes.

  For surjectivity, suppose $(\delta, \pi)$ is a representation of a strong real form of $G^{\Gamma}$. Using the surjectivity of (\ref{5.2map}) and Lemma \ref{idiso} we may take $\delta = z(\sigma) \delta_{q}$ for a rigid inner twist $(\mathrm{id}, z)$ in which the codomain of $\mathrm{id}$ has $\mathbb{R}$-structure defined by $\mathrm{Int}(z(\sigma) \delta_{q})$. This ensures that $(\mathrm{id}, z, \pi)$ is a representation of a rigid inner twist of $(G, \delta_{q})$.  It is obvious that $(\mathrm{id}, z, \pi)$ maps to $(\delta, \pi)$ under (\ref{repmap}).

  For injectivity, we may assume without loss of generality that $(\mathrm{id},z, \pi)$ and $(\mathrm{id},z_{1}, \pi_{1})$ map to $(\delta, \pi)$ and $(\delta_{1}, \pi_{1})$ respectively.  If there exists $g \in G$ such that
  $$(\delta_{1}, \pi_{1}) = (g \delta g^{-1}, \pi \circ \mathrm{Int}(g^{-1}))$$
then it is straightforward to show that $(\mathrm{id}, g)$ is an isomorphism  between $(\mathrm{id},z, \pi)$ and $(\mathrm{id},z_{1}, \pi_{1})$.
\end{proof}

\subsection{Pairings for real algebraic tori}

Pairings between equivalence classes of strong real forms of tori and
certain component groups appear in both \cite{abv} and \cite{kal}.  We
describe the pairings first in the context of \cite{abv} and then show
that the pairings of \cite{kal} agree with them when working over the real numbers.

Let $T$ be a complex algebraic torus defined over $\mathbb{R}$, \emph{i.e.} with a $\Gamma$-action.  We define the extended group
$$T = T \rtimes \langle \delta_{q} \rangle, \ \mathcal{W} = \{\delta_{q} \},$$
where $\delta_{q}^{2} = 1$ and $\delta_{q}$ acts on $T$ by the non-trivial element $\sigma  \in \Gamma$.  Set $T(\mathbb{R}) = T(\mathbb{R}, \delta_{q})$.  The strong real forms of $T$ are the elements of the form $t \delta_{q} \in T^{\Gamma}$ for which
$$(t \delta_{q})^{2} =  t \delta_{q}(t)\, \delta_{q}^{2} = t \delta_{q}(t)$$
is of finite order.
Clearly, $\mathrm{Int}(t\delta_{q}) = \mathrm{Int}(\delta_{q})$ and so
$$T(\mathbb{R}, t \delta_{q}) = T(\mathbb{R})$$
for any strong real form of $T^{\Gamma}$.

Let ${^\vee}T^{\Gamma} =  {^\vee}T \rtimes \langle {^\vee}\delta_{q} \rangle$ be the L-group of $T$.  Just as for the extended group, conjugation by any element ${^\vee}T^{\Gamma}- {^\vee}T$ defines the same automorphism ${^\vee}\theta = \mathrm{Int}( {^\vee}\delta_{q} )$ of ${^\vee}T$.  The transpose ${^\vee}\theta^{\intercal}$ of ${^\vee}\theta$ defines an automorphism of $X^{*}({^\vee}T) = X_{*}(T)$.  The Lie algebra of $T$ may be identified with $(X_{*}(T) \otimes \mathbb{C})/X_{*}(T)$ via the map $\exp(2 \uppi i \cdot)$.  Under this identification, the differential of the action of $\sigma \in \Gamma$ is equal to ${^\vee}\theta^{\intercal}$ (\cite{abv}*{Lemma 9.9 (d)}).

According to \cite{abv}*{Lemma 9.9, Proposition 9.10}, the map
\begin{equation}
  \label{9.10map}
  \lambda_{1} \mapsto \exp(\uppi i \lambda_{1})\, (1-\sigma)T, \quad \lambda_{1} \in (X_{*}(T) \otimes \mathbb{Q})^{-\sigma}/(1-\sigma) X_{*}(T)
\end{equation}
defines an isomorphism onto the group
$$\{t \in T: t \delta_{q}(t) \mbox{ is of finite order}\}/ (1-\sigma)T.$$
In consequence, the map
\begin{equation}
  \label{9.10map1}
\lambda_{1} \mapsto \exp(\uppi i \lambda_{1}) \, \delta_{q}, \quad
\lambda_{1} \in (X_{*}(T) \otimes \mathbb{Q})^{-\sigma}/(1-\sigma)
X_{*}(T)
\end{equation}
passes to a bijection onto the set of equivalence classes of strong real forms of $T^{\Gamma}$.

The restriction of (\ref{9.10map}) to $X_{*}(T)^{-\sigma}$ produces an isomorphism from
$$X_{*}(T)^{-\sigma}/ (1-\sigma)X_{*}(T) $$
onto
$$\{t \in T: t\delta_{q}(t) = 1\}/(1-\sigma)T = \{t \in T: t\sigma(t) = 1\}/(1-\sigma)T \cong H^{1}(\Gamma, T)$$
which is the set of equivalence classes of pure real forms of
$T^{\Gamma}$.  Tate-Nakayama duality identifies the latter group with
the characters of the component group
$${^\vee}T^{{^\vee}\theta}/({^\vee}T^{{^\vee}\theta})^{0}$$
(\cite{kottwitz84}*{(3.3.1)}).  In summary,
\begin{align}
 \label{pureforms}
\mbox{Equivalence classes of pure real forms}
 &\cong  X_{*}(T)^{-\sigma}/ (1-\sigma)X_{*}(T)\\
\nonumber &\cong \mathrm{Hom} \left( {^\vee}T^{{^\vee}\theta}/({^\vee}T^{{^\vee}\theta})^{0}, \mathbb{C}^{\times} \right).
 \end{align}
 If one identifies the equivalence classes of pure real forms with the groups on the right, then one may equally well say that there is a perfect pairing between ${^\vee}T^{{^\vee}\theta}/({^\vee}T^{{^\vee}\theta})^{0}$ and the set of equivalence classes of pure real forms.
 
We would like a similar pairing for equivalence classes of strong real forms.  Towards this end, we already have  isomorphism (\ref{9.10map1}).   For the second isomorphism  we need to introduce the \emph{algebraic universal covering} ${^\vee}T^{\mathrm{alg}}$ of ${^\vee}T$,  which is the projective limit of all finite coverings of ${^\vee}T$.  The algebraic universal covering is part of a short exact sequence
\begin{equation}
  \label{covseq}
  1 \rightarrow \uppi_{1}({^\vee}T)^{\mathrm{alg}} \rightarrow {^\vee}T^{\mathrm{alg}} \rightarrow {^\vee}T \rightarrow 1
  \end{equation}
in which $\uppi_{1}({^\vee}{T})^{\mathrm{alg}}$ is a profinite group (\cite{abv}*{(5.10)}). According to \cite{abv}*{Proposition 9.8 (c)} 
\begin{align}
\label{strongforms}
\mbox{Equivalence classes of strong real forms}
 &\cong  (X_{*}(T) \otimes \mathbb{Q})^{-\sigma}/ (1-\sigma)X_{*}(T)\\
\nonumber &\cong \mathrm{Hom} \left( ({^\vee}T^{{^\vee}\theta})^{\mathrm{alg}}/(({^\vee}T^{{^\vee}\theta})^{\mathrm{alg}} )^{0}, \mathbb{C}^{\times} \right).
 \end{align}
 Here, the group $({^\vee}T^{{^\vee}\theta})^{\mathrm{alg}}$ is the
 preimage of ${^\vee}T^{{^\vee}\theta} \subset {^\vee}T$ in
 (\ref{covseq}).  The isomorphisms of (\ref{strongforms}) amount to a perfect pairing between 
$({^\vee}T^{{^\vee}\theta})^{\mathrm{alg}}/(({^\vee}T^{{^\vee}\theta})^{\mathrm{alg}} )^{0}$ and the set of equivalence classes of strong real forms.

Our next task is to determine the correct pairing for equivalence
classes of strong real forms of type $J \subset T^{\mathrm{tor}}$.  In
order to compare with rigid inner twists, we restrict to the case that
$J$ is finite and defined over $\mathbb{R}$.  Under this assumption,
equivalence classes of strong real forms of type $J$ are a subset of
(\ref{strongforms}) which is related to an isogeny of $T$.  Set
$\bar{T} = T/J$.  As $J$ is defined over $\mathbb{R}$, it is stable under
$\delta_{q}$, and so we may define $\bar{T}^{\Gamma} = \bar{T} \rtimes \langle \delta_{q} \rangle$.  It is an algebraic torus which is related to $T$ through the isogeny
$$\iota: T \rightarrow T/J.$$
The dual isogeny
$$\hat{\iota} : {^\vee}\bar{T} \rightarrow {^\vee}T$$
yields an algebraic cover
$$1 \rightarrow \ker{\hat{\iota}} \rightarrow {^\vee}\bar{T} \rightarrow {^\vee}T \rightarrow 1.$$
As $\ker{\hat{\iota}}$ is finite, there is a surjective map
$$\uppi_{1}({^\vee}T)^{\mathrm{alg}} \rightarrow \ker{\hat{\iota}}.$$
This surjection induces the commutative diagram
$$\xymatrix{1 \ar[r]   & \uppi_{1}({^\vee}T)^{\mathrm{alg}} \ar[r] \ar[d] & {^\vee}T^{\mathrm{alg}} \ar[r]  \ar[d] & {^\vee}T \ar[r] \ar[d]^{\cong} & 1\\
1 \ar[r] & \ker{\hat{\iota}} \ar[r] & {^\vee}\bar{T}\ar[r]  & {^\vee}T \ar[r] & 1}$$
with surjective columns and exact rows (\cite{abv}*{10.10}).
\begin{lem}
\label{hatJ}
 The kernel of $\hat{\iota}$ is isomorphic to $\hat{J} = \mathrm{Hom}(J, \mathbb{C}^{\times})$.
 \end{lem}
\begin{proof}
 The dual of the short exact sequence
  $$1 \rightarrow J \rightarrow T \stackrel{\iota}{\rightarrow} \bar{T} \rightarrow 1$$
  yields the short exact sequence
  $$1 \rightarrow X^{*}(\bar{T}) \stackrel{\iota^{*}}{\rightarrow} X^{*}(T) \rightarrow \hat{J} \rightarrow 1$$
(\cite{BorelLAG}*{Section 8.12}).
This implies in turn that  $\hat{J} \cong X^{*}(T)/ \iota^{*} X^{*}(\bar{T})$ and
\begin{equation}
    \label{Jiso}
    J \cong \mathrm{Hom}\left( X^{*}(T)/ \iota^{*}X^{*}(\bar{T}),\,  \mathbb{C}^{\times} \right).
\end{equation}
Identical reasoning applied to the dual isogeny implies that
$$  \ker \hat{\iota} \cong \mathrm{Hom}\left( X_{*}(\bar{T})/ \hat{\iota}^{*}X_{*}(T),\,  \mathbb{C}^{\times} \right),$$
where 
\begin{equation}
\label{iotamap}
\hat{\iota}^{*} : X_{*}(T) \rightarrow X_{*}(\bar{T})
\end{equation}
is the transpose of $\iota^{*}$.  
The perfect $\mathbb{Z}$-pairing between $X_{*}(T)$ and $X^{*}(T)$ allows us to rewrite this isomorphism as
$$  \ker \hat{\iota} \cong \mathrm{Hom}\left( \mathrm{Hom}(X^{*}(\bar{T}),\mathbb{Z})/ \hat{\iota}^{*} \mathrm{Hom}(X^{*}(T),\mathbb{Z}),\,  \mathbb{C}^{\times} \right).$$
Looking back to (\ref{Jiso}) and applying duality, we are reduced to proving 
\begin{equation}
\label{dJiso}
  X^{*}(T)/ \iota^{*}X^{*}(\bar{T}) \cong \mathrm{Hom}\left( \mathrm{Hom}(X^{*}(\bar{T}),\mathbb{Z})/ \hat{\iota}^{*}\mathrm{Hom}(X^{*}(T),\mathbb{Z}),\,  \mathbb{C}^{\times} \right).
\end{equation}
This is a statement about a quotient of $X^{*}(T) \cong \mathbb{Z}^{n}$ by a full-rank sublattice.   By choosing $\mathbb{Z}$-bases for $X^{*}(T)$ and $X^{*}(\bar{T})$, we identify both groups with $\mathbb{Z}^{n}$ and $\iota^{*}$ as an endomorphism of $\mathbb{Z}^{n}$.   It is an elementary exercise to show that  $\mathbb{Z} \cong \mathrm{Hom}(\mathbb{Z},\mathbb{Z})$, that $\iota^{*}$ has the same invariant factors, $m_{1}, \ldots, m_{n}$, as the transpose of $\iota^{*}$, and that the invariant factors do not depend on the choice of bases.  Consequently, isomorphism (\ref{dJiso}) is equivalent to
$$\oplus_{j=1}^{n} \mathbb{Z}/ m_{j}\mathbb{Z} \cong \mathrm{Hom}(\oplus_{j=1}^{n} \mathbb{Z}/ m_{j}\mathbb{Z}, \, \mathbb{C}^{\times}),$$
which follows from the well-known isomorphism
$$ \mathbb{Z}/ m_{j}\mathbb{Z} \cong \mathrm{Hom}( \mathbb{Z}/ m_{j}\mathbb{Z}, \, \mathbb{C}^{\times}).$$
\end{proof}
This lemma allows us to denote ${^\vee}\bar{T}$  of
\cite{kal}*{Section 5.3} by ${^\vee}T^{\hat{J}}$. In this way
$$\xymatrix{1 \ar[r]   & \uppi_{1}({^\vee}T)^{\mathrm{alg}} \ar[r] \ar[d] & {^\vee}T^{\mathrm{alg}} \ar[r]  \ar[d] & {^\vee}T \ar[r] \ar[d]^{\cong} & 1\\
1 \ar[r] & \hat{J} \ar[r] & {^\vee}T^{\hat{J}}  \ar[r]  & {^\vee}T
\ar[r] & 1}$$
which conforms with the notation of \cite{abv}*{(5.13)}.

We shall see in a moment how the set of equivalence classes of strong real forms of $T$ of type $J$ is associated to $X_{*}(\bar{T})$.  The precise statement depends on the injection (\ref{iotamap}), which is the bridge between the two relevant groups.  It is possible to identify $X_{*}(T)$ with its image under (\ref{iotamap}) and this is done in \cite{kal}*{Section 4.1}.  Alternatively, one may extend $\hat{\iota}^{*}$ to a vector space isomorphism
$$\hat{\iota}^{*}_{\mathbb{Q}}:  X_{*}(T) \otimes \mathbb{Q} \rightarrow X_{*}(\bar{T}) \otimes \mathbb{Q},$$
and identify $X_{*}(\bar{T})$ with its preimage under
$\hat{\iota}^{*}_{\mathbb{Q}}$ in $X_{*}(T) \otimes \mathbb{Q}$.  We
then set 
\begin{equation}
  \label{preim}
\frac{X_{*}(\bar{T})^{-\sigma}}{(1-\sigma) X_{*}(T)} = 
(\hat{\iota}_{\mathbb{Q}}^{*})^{-1} \left(
\frac{X_{*}(\bar{T})^{-\sigma}}{(1-\sigma) X_{*}(\bar{T})} \right)
\subset  \frac{(X_{*}(T) \otimes \mathbb{Q})^{-\sigma}}{(1-\sigma) X_{*}(T)}.
\end{equation}
We use this identification below.
\begin{lem}
\label{typeJiso}
The map (\ref{9.10map1}) sends $X_{*}(\bar{T})^{-\sigma}/(1-\sigma)X_{*}(T)$ bijectively onto the set of equivalence classes of strong real forms of $T^{\Gamma}$ of type $J$. 
\end{lem}
\begin{proof}
Suppose $\bar{t}\delta_{q}$ is a strong real form of
$\bar{T}^{\Gamma}$ where $\bar{t} = tJ$ for some $t \in T$.  Then
$\bar{t} \delta_{q}(\bar{t})$ is of finite order.  This in turn
implies $(t \delta_{q} (t))^{m} \in J$ for some $m \geq 1$, $t
\delta_{q}(t)$ has finite order, and $t \delta_{q}$ is a strong real
form of $T^{\Gamma}$.  Combining these
observations with (\ref{9.10map1}), we  obtain the commutative diagram
$$\xymatrix{ \frac{(X_{*}(T) \otimes \mathbb{Q})^{-\sigma}}{(1-\sigma) X_{*}(T)}  \ar@{<->}[r]
\ar@{->>}[d]_{\hat{\iota}^{*}_{\mathbb{Q}}} & \mbox{equivalence classes strong real forms } 
T^{\Gamma} \ar@{->>}[d]^{/J}\\ 
\frac{(X_{*}(\bar{T}) \otimes \mathbb{Q})^{-\sigma}}{(1-\sigma)
  X_{*}(\bar{T})} \ar@{<->}[r] & \mbox{equivalence classes strong real
  forms } \bar{T}^{\Gamma} }$$
  Arguing as above one sees  that the strong real
forms of $T^{\Gamma}$ of type $J$ map onto the pure real forms of
$\bar{T}^{\Gamma}$ upon taking the quotient by $J$. Therefore the previous
diagram and  (\ref{preim}) imply
$$\xymatrix{
  \frac{X_{*}(\bar{T})^{-\sigma}}{(1-\sigma) X_{*}(T)}  \ar@{<->}[r]
\ar@{->>}[d]_{\hat{\iota}^{*}_{\mathbb{Q}}} & \mbox{equivalence classes strong real forms type } 
J \ar@{->>}[d]^{/J}\\ 
\frac{X_{*}(\bar{T})^{-\sigma}}{(1-\sigma) X_{*}(\bar{T})} \ar@{<->}[r] & \mbox{equivalence classes pure real forms } \bar{T}^{\Gamma} }$$
The top horizontal arrow of the latter diagram is the assertion of the lemma.
\end{proof}

We now have a refinement of (\ref{strongforms})
\begin{align}
\label{Jforms}
\nonumber & \mbox{Equivalence classes of strong real forms of type } J\\
 &\cong  X_{*}(\bar{T})^{-\sigma}/ (1-\sigma)X_{*}(T)\\
\nonumber &\cong \mathrm{Hom} \left( ({^\vee}T^{{^\vee}\theta})^{\hat{J}}/ 
(({^\vee}T^{{^\vee}\theta})^{\hat{J}} )^{0}, \mathbb{C}^{\times} \right)
 \end{align}
in which the bottom isomorphism is given by \cite{abv}*{Theorem 10.11}
after writing ${^\vee}\bar{T}$ as ${^\vee}T^{\hat{J}}$.  This is equivalent to a
 perfect pairing between $({^\vee}T^{{^\vee}\theta})^{\hat{J}}/(({^\vee}T^{{^\vee}\theta})^{\hat{J}} )^{0}$ and the equivalence classes of real forms of type $J$.

An application of \cite{kal}*{Corollary 5.4} to $T$ produces a perfect pairing resembling the one we 
have just established.  There are two apparent differences between the pairings.  The first 
difference is the use of the torsion subgroup of $X_{*}(\bar{T})/ (1-\sigma)X_{*}(T)$ in place of $X_{*}
(\bar{T})^{-\sigma}/ (1-\sigma)X_{*}(T)$ above.  \cite{kal}*{Fact 4.1} states that the two groups are 
actually equal, and  this may be proven by decomposing a torsion element into $+1$ and $-1$ 
eigenvectors with respect to $\sigma$ in $X_{*}(T) \otimes \mathbb{C}$.  

The other apparent difference in the pairings is the isomorphism between 
$X_{*}(\bar{T})^{-\sigma}/ (1-\sigma)X_{*}(T)$ and the equivalence classes of strong real forms of
type $J$.  The isomorphism is given in \cite{kal}*{Theorem 4.8} and is rather intricate over p-adic fields.  
Fortunately, over the real numbers the setup simplifies considerably (\cite{kal}*{Theorem 5.2}), and 
one may prove that the isomorphism is equal to (\ref{9.10map}).  Without delving too far into 
the details, the proof runs as follows.  In \cite{kal}*{Section 4.6}
the element $\lambda_{1} \in X_{*}(\bar{T})^{-\sigma}$ (\emph{cf.}
(\ref{9.10map1})) is assigned to the the strong real form $t\delta_{q}$ in which
\begin{equation}
\label{cupprod}
t = \left( l_{k}c_{k} \sqcup_{\mathbb{C}/\mathbb{R}} \,k!  \lambda_{1} \right) (\sigma).
\end{equation}
In this equation, $\sigma \in \Gamma$ is non-trivial, $k>0$ is any integer divisible by $|J|$, $l_{k}$ is 
the (usual) $k!$-th root function on $\mathbb{C}^{\times}$, $c_{k} \in Z^{2}(\Gamma, \mathbb{C}
^{\times})$ is the 2-cocycle defining $W_{\mathbb{R}}$, and 
$\sqcup_{\mathbb{C}/\mathbb{R}}$ is a cup product defined in \cite{kal}*{Section 4.3}.  The formula 
for the computation of (\ref{cupprod}) appears at the end of \cite{kal}*{Section 4.3} and one may 
compute that it is 
$$k! \lambda_{1}(\exp(\uppi i/k!)) = \lambda_{1}( \exp(\uppi i)).$$
Since we are regarding $X_{*}(\bar{T})$ as a submodule of $X_{*}(T) \otimes \mathbb{Q}$, this 
element is regarded as an element in $T$.  It is an elementary exercise to prove that 
$\lambda_{1}( \exp(\uppi i))$ coincides with $\exp( \uppi i \lambda_{1})$ under the identification we have made 
for the Lie algebra of $T$.  In consequence, the map $\lambda_{1} \mapsto  \lambda_{1}( \exp(\uppi i))$ 
induces the same map as (\ref{9.10map}), and the first isomorphism of (\ref{Jforms}) coincides with 
Kaletha's.  
The second isomorphism of (\ref{Jforms}) also coincides with Kaletha's (\cite{kal}*{Proposition 5.3}), 
so that  the perfect  pairings between $({^\vee}T^{{^\vee}\theta})^{\hat{J}}/(({^\vee}T^{{^\vee}
\theta})^{\hat{J}} )^{0}$ and the equivalence classes of real forms of type $J$ in (\ref{Jforms}) and in 
\cite{kal}*{Corollary 5.4} are the same.

\section{L-packets in Adams-Barbasch-Vogan}
\label{abvlp}

In this section we review the definition of L-packets as given in \cite{abv}.  In this framework  the representations in an L-packet are representations of strong real forms, and usually run over different real forms of a given group.   One may calibrate the size of an L-packet by specifying a subgroup $J \subset Z(G)^{\mathrm{tor}}$ and considering only strong real forms of type $J$.  We first complete the description of L-packets over \emph{all} strong real forms, \emph{i.e.} where $J = Z(G)^{\mathrm{tor}}$.  Towards the end of the section we shall discuss how the construction is affected in taking $J$ to be a smaller subgroup.  

An L-packet is determined by a (${^\vee}G$-conjugacy class of an) L-homomorphism as in (\ref{Lhom}).
The (equivalence classes of) irreducible representations which appear
in the L-packet $\Pi^{\mathrm{ABV}}_{\phi}$ are parameterized by
characters of an abelian group.  For this we need the algebraic universal covering ${^\vee}G^{\mathrm{alg}}$ of ${^\vee}G$, 
\begin{equation}
  \label{covseq1}
  1 \rightarrow \uppi_{1}({^\vee}G)^{\mathrm{alg}} \rightarrow {^\vee}G^{\mathrm{alg}} \rightarrow {^\vee}G \rightarrow 1
  \end{equation}
in which $\uppi_{1}({^\vee}{G})^{\mathrm{alg}}$ is a profinite group.  Let ${^\vee}G_{\phi}$ be the isotropy subgroup of ${^\vee}G$ acting on $\phi$ under conjugation, \emph{i.e.} the centralizer of $\phi(W_{\mathbb{R}})$ in ${^\vee}G$.   Let ${^\vee}G_{\phi}^{\mathrm{alg}}$ be the preimage of ${^\vee}G_{\phi}$ under (\ref{covseq1}), and let
$${^\vee}G_{\phi}^{\mathrm{alg}}/ ({^\vee}G_{\phi}^{\mathrm{alg}})^{0}$$
be the component group of ${^\vee}G_{\phi}^{\mathrm{alg}}$.  The
component group  is abelian (\cite{abv}*{\emph{p.} 61}). Let
$({^\vee}G_{\phi}^{\mathrm{alg}}/
({^\vee}G_{\phi}^{\mathrm{alg}})^{0})^{\wedge}$ be its group of characters.  The elements of $\Pi^{\mathrm{ABV}}_{\phi}$ are parameterized by the elements $\tau \in ({^\vee}G_{\phi}^{\mathrm{alg}}/ ({^\vee}G_{\phi}^{\mathrm{alg}})^{0})^{\wedge}$ and most of this section is devoted to the description of this parameterization.

For convenience, we fix  Borel subgroups, $B \subset G$ and
$\mathcal{B} \subset {^\vee}G$, and maximal tori $T \subset B$ and
$\mathcal{T} \subset \mathcal{B}$.  We may and do choose these
subgroups  to be stable under the $\Gamma$-actions of $G{^\Gamma}$ and
${^\vee}G^{\Gamma}$ respectively.  Write $W_{\mathbb{R}} =
  \mathbb{C}^{\times} \cup j\mathbb{C}^{\times}$.  
We may assume that $\phi: W_{\mathbb{R}} \rightarrow {^\vee}G^{\Gamma}$ is an L-homomorphism such that
\begin{equation}
\label{lhom}
\phi(z) = z^{\lambda} \bar{z}^{\mathrm{Ad}(\phi(j))\lambda} \in \mathcal{T}, \quad z \in \mathbb{C}^{\times}
\end{equation}
for some $\lambda \in X_{*}(\mathcal{T}) \otimes \mathbb{C}$.  We identify the Lie algebra of $\mathcal{T}$ with $(X_{*}(\mathcal{T}) \otimes \mathbb{C})/ X_{*}(\mathcal{T})$ via the map $\exp(2\uppi i\,  \cdot)$.  Set $y = \exp(\uppi i \lambda)  \phi(j)$.

The pair $(\lambda,y)$ attached to $\phi$ allows us to choose another maximal torus ${^d}T \subset {^\vee}G$ in the following manner.  First, define the Levi subgroup
\begin{equation}
\label{dL}
{^d}L =  {^d}L(\lambda, y\cdot \lambda) = \{ g \in {^\vee}G: \mathrm{Int}(g)\circ\lambda = \lambda, \ \mathrm{Int}(g)\mathrm{Int}(y)\circ \lambda = \mathrm{Int}(y) \circ \lambda\} \subset {^\vee}G.
\end{equation}
Evidently, $\mathcal{T}$ is a maximal subtorus of ${^d}L$.  Since $y^{2} = \exp(2 \uppi i \lambda)$ (\cite{abv}*{Proposition 5.6}), it follows  that $y^{2}$ is central in ${^d}L$.  Therefore, in setting  ${^\vee}\theta = \mathrm{Int}(y)$ and restricting it to ${^{d}L}$, we obtain an involutive automorphism of ${^d}L$.  By \cite{abv}*{Lemma 12.10} there exists a maximal torus ${^d}T \subset {^d}L$ with the property that
$({^d}\mathfrak{t})^{-{^\vee}\theta}$ is a maximal semisimple abelian
subalgebra of $({^d}\mathfrak{l})^{-{^\vee}\theta}$.  The torus
${^d}T$ is unique up to conjugation by the identity component of the
subgroup ${^d}K  = ({^d}L)^{{^\vee}\theta}$.  One may interpret such a
maximal torus as a maximally split torus of ${^d}L$.

As ${^d}L$ is a Levi subgroup of ${^\vee}G$, the torus ${^d}T$ is maximal in ${^\vee}G$ as well.   Let ${^d}T^{\Gamma}$  be the group generated by $y$ and ${^d}T$.  It is not difficult to see that $\phi$ takes values in ${^d}T^{\Gamma}$ (\cite{abv}*{(12.11)}).  Define
$$\phi_{d}: W_{\mathbb{R}} \rightarrow {^d}T^{\Gamma}$$
to be the resulting L-homomorphism for ${^d}T^{\Gamma}$.

This describes the passage from $\phi$, a homomorphism associated to ${^\vee}G$, to $\phi_{d}$, a 
homomorphism associated to ${^d}T$.  We wish to accomplish the same thing for $\tau \in 
({^\vee}G_{\phi}^{\mathrm{alg}}/ ({^\vee}G_{\phi}^{\mathrm{alg}})^{0})^{\wedge}$.  The analogue of the component group for ${^d}T$ is defined by 
setting $({^d}T^{{^\vee}\theta})^{\mathrm{alg}, {^\vee}G}$ to be the preimage of ${^d}T^{{^\vee}
  \theta}$ in (\ref{covseq1}) and then considering its component group
\begin{equation}
\label{Tcomp}
  ({^d}T^{{^\vee}\theta})^{\mathrm{alg}, {^\vee}G}/
(({^d}T^{{^\vee}\theta})^{\mathrm{alg}, {^\vee}G})^{0}.
\end{equation}
The inclusion of $({^d}T^{{^\vee}\theta})^{\mathrm{alg}, {^\vee}G}$ into ${^\vee}G_{\phi}^{\mathrm{alg}}$ induces a surjection
$$({^d}T^{{^\vee}\theta})^{\mathrm{alg}, {^\vee}G}/ (({^d}T^{{^\vee}\theta})^{\mathrm{alg}, {^\vee}G})^{0} \rightarrow {^\vee}G_{\phi}^{\mathrm{alg}}/ ({^\vee}G_{\phi}^{\mathrm{alg}})^{0}$$
(\cite{abv}*{(12.11)(e)}).  The  kernel of this surjection is generated by the cosets of the elements
\begin{equation}
\label{alphaone}
{^\vee}\alpha(-1) \in ({^d}T^{{^\vee}\theta})^{\mathrm{alg}, {^\vee}G}
\end{equation}
(\cite{abv}*{(12.6)(c), (12.11)(f)}), where $\alpha \in R({^d}L,{^d}T)$ satisfies $\alpha \circ {^\vee}\theta = -\alpha$ (\emph{i.e.} $\alpha$ is ${^\vee}\theta$-\emph{real}).  Note that the definition of ${^d}L$ forces these roots to satisfy $\langle \lambda, \alpha \rangle = 0$. In consequence, we call these the \emph{real} $\phi$-\emph{singular} roots of $R({^\vee}G,{^d}T)$ and denote the set of them by $R_{\mathbb{R}}({^d}L,{^d}T)$.  At this point, we abandon the ${^\vee}G$ in superscript and summarize by writing
\begin{equation}
  \label{compiso}
({^d}T^{{^\vee}\theta})^{\mathrm{alg}}/
  (({^d}T^{{^\vee}\theta})^{\mathrm{alg}})^{0}\langle
        {^\vee}\alpha(-1): \alpha \in R_{\mathbb{R}}({^d}L,{^d}T)
        \rangle \cong {^\vee}G_{\phi}^{\mathrm{alg}}/
        ({^\vee}G_{\phi}^{\mathrm{alg}})^{0}.
\end{equation}        
This isomorphism induces an isomorphism of the character groups.  We
define $\tau_{d}$ to be the image of $\tau \in
({^\vee}G_{\phi}^{\mathrm{alg}}/
({^\vee}G_{\phi}^{\mathrm{alg}})^{0})^{\wedge}$ under the isomorphism
of character groups.  We shall often identify $\tau_{d}$ with a character of $({^d}T^{{^\vee}\theta})^{\mathrm{alg}}/ (({^d}T^{{^\vee}\theta})^{\mathrm{alg}})^{0}$ which is trivial on the cosets of ${^\vee}\alpha(-1)$,  $\alpha \in R_{\mathbb{R}}({^d}L,{^d}T)$.
So far we have described an assignment
\begin{equation}
  \label{1map}
  (\phi, \tau) \mapsto (\phi_{d},\tau_{d}).
\end{equation}
The pair on the right is a \emph{complete Langlands parameter of the
  Cartan subgroup} ${^d}T^{\Gamma}$ with respect to
${^\vee}G^{\Gamma}$ (\cite{abv}*{Definition 12.4}).

Using \cite{abv}*{Proposition  13.10 (b)} we obtain a Cartan subgroup
$T^{\Gamma} \subset G^{\Gamma}$ (\cite{abv}*{Definition 12.1}) related
to ${^d}T$.  First off, $T^{\Gamma}$ is a weak extended group of $T$,
which is itself a maximal torus of $G$.  The proposition also provides
an element $\delta_{d} \in T^{\Gamma}$ such that $\delta_{d} = g
\delta_{q} g^{-1}$ for some $g \in G$ (\ref{extG}). This implies that
$(\delta_{d}, g^{-1}N g, \chi \circ \mathrm{Int}(g)) \in \mathcal{W}$
and that $G(\mathbb{R}, \delta_{d})$ is a quasisplit real form of $G$.
In addition, the proposition provides an isomorphism ${^\vee}T \cong
{^d}T$ which identifies ${^d}T^{\Gamma}$ as an \emph{E-group} of $T$
(\cite{abv}*{Definition 4.6}).
The classical Langlands correspondence matches $\phi_{d}$ with an irreducible representation $\pi_{d}$ of $T(\mathbb{R}) = T(\mathbb{R},\delta_{d})$.  

According to (\ref{strongforms}), there is a map from
$ ({^d}T^{{^\vee}\theta})^{\mathrm{alg}}/
(({^d}T^{{^\vee}\theta})^{\mathrm{alg}})^{0}$ to equivalence classes of strong
real forms of $T^{\Gamma}$.  According to \cite{abv}*{Proposition
  13.12 (a)},  the image of (\ref{Tcomp}) under this map actually lies in the equivalence classes of strong real forms of $G^{\Gamma}$.   The map is denoted by
$$\tau_{d} \mapsto \delta(\tau_{d})$$
in \cite{abv}*{Proposition 13.12}, and is chosen to send the trivial character  to the equivalence class of $\delta_{d}$, \emph{i.e.} $\delta_{d} = \delta(1)$. We may progress from (\ref{1map}) by writing
\begin{equation}
  \label{2map}
  (\phi, \tau) \mapsto (\phi_{d}, \tau_{d}) \mapsto (\pi_{d}, \delta(\tau_{d})).
\end{equation}
For the  final step in this sequence of maps, observe that $T(\mathbb{R}, \delta(\tau_{d})) = T(\mathbb{R})$.  Hence,  $\pi_{d}$ is a representation of 
\begin{equation}
\label{inc}
T(\mathbb{R}, \delta(\tau_{d})) \subset G(\mathbb{R}, \delta(\tau_{d})).
\end{equation}
This representation may be induced, cohomologically and parabolically, to a \emph{standard} 
representation $M(\pi_{d}, \delta(\tau_{d}))$ of $G(\mathbb{R}, \delta(\tau_{d}))$ (\cite{abv}*{Section 
11, Theorem 12.3, Proposition 13.2}).  The standard representation has a unique Langlands 
quotient, which is an irreducible representation $\pi_{\tau_{d}}$ of $G(\mathbb{R}, \delta(\tau_{d}))$.  The final 
step to be added to (\ref{2map}) is
\begin{equation}
\label{abvcorr}
(\phi, \tau) \mapsto (\phi_{d}, \tau_{d}) \mapsto (\pi_{d}, \delta(\tau_{d})) \mapsto (\pi_{\tau_{d}}, 
\delta(\tau_{d})).
\end{equation}
The L-packet of $\phi$ is defined to consist of the distinct equivalence classes of the representations 
of strong real forms $(\pi_{\tau_{d}}, \delta(\tau_{d}))$, $\tau \in ({^\vee}G_{\phi}^{\mathrm{alg}}/ ({^\vee}G_{\phi}^{\mathrm{alg}})^{0})^{\wedge}$.  Ignoring 
any indication of equivalences in our notation, we simply  write
$$\Pi^{\mathrm{ABV}}_{\phi} = \left\{ (\pi_{\tau_{d}}, \delta(\tau_{d})): \tau  \in ({^\vee}G_{\phi}^{\mathrm{alg}}/ ({^\vee}G_{\phi}^{\mathrm{alg}})^{0})^{\wedge} \right\}.$$

We now consider a torsion subgroup $J \subset Z(G)^{\mathrm{tor}}$ and the parameterization of the smaller L-packet defined by
$$\Pi^{\mathrm{ABV}}_{\phi,J} = \left\{ (\pi_{\tau_{d}}, \delta(\tau_{d})): \tau  \in ({^\vee}G_{\phi}^{\mathrm{alg}}/ ({^\vee}G_{\phi}^{\mathrm{alg}})^{0})^{\wedge} ,\ \delta(\tau_{d})^{2} \in J \right\}.$$
There is a natural isomorphism between $Z(G)^{\mathrm{tor}}$ and the group $(\uppi_{1}(^{\vee}G)^{\mathrm{alg}})^{\wedge}$ of continuous characters of $\uppi_{1}(^{\vee}G)^{\mathrm{alg}}$, the 
\emph{Pontryagin dual} of $\uppi_{1}(^{\vee}G)^{\mathrm{alg}}$ (\cite{abv}*{Lemma 10.9 (a)}).  We 
may therefore identify $J$ with a subgroup of $(\uppi_{1}(^{\vee}G)^{\mathrm{alg}})^{\wedge}$ and write
$$J \hookrightarrow (\uppi_{1}(^{\vee}G)^{\mathrm{alg}})^{\wedge}$$
Applying Pontryagin duality  yields a surjection
$$\uppi_{1}(^{\vee}G)^{\mathrm{alg}} \rightarrow \hat{J}$$
(\cite{Folland}*{Theorem 4.31, Corollary 4.41}).  We form a profinite covering
\begin{equation}
  \label{Jhatcover}
  1 \rightarrow \hat{J} \rightarrow {^\vee}G^{\hat{J}} \rightarrow {^\vee}G \rightarrow 1
\end{equation}
(\cite{abv}*{(10.10)}), and repeat the earlier constructions replacing ``alg'' in superscript everywhere 
with $\hat{J}$.  In particular, the component group becomes
$${^\vee}G_{\phi}^{\hat{J}}/ ({^\vee}G_{\phi}^{\hat{J}})^{0}.$$
According to \cite{abv}*{Theorem 10.11}
\begin{equation}
\label{JLpacket}
\Pi^{\mathrm{ABV}}_{\phi,J} = \left\{ (\pi_{\tau_{d}}, \delta(\tau_{d})): \tau  \in ({^\vee}G_{\phi}^{\hat{J}}/ ({^\vee}G_{\phi}^{\hat{J}})^{0})^{\wedge} 
\right\}.
\end{equation}

In anticipation of the comparison with the  L-packets of the next section, let us rephrase inclusion 
(\ref{inc}) in the language of embeddings when the strong real forms
$\delta(\tau_{d})$ are of type $J$,  where $J \subset Z(G)$ is a finite  subgroup defined over $\mathbb{R}$.  As in Lemma \ref{idiso}, the 
group $G(\mathbb{R}, \delta(\tau_{d}))$ is the group of real points of a rigid inner twist $(\mathrm{id},z_{\tau_{d}})$ for 
which $z_{\tau_{d}}(\sigma) = \delta(\tau_{d}) \delta_{q}$.  Inclusion (\ref{inc}) may be written as an $\mathbb{R}$-embedding
$$\eta_{\tau_{d}}: T \rightarrow G$$
in which the $\mathbb{R}$-structure of $G$ is defined by $\delta(\tau_{d})$ (or equivalently $z_{\tau_{d}}$).     It 
makes sense to call $\eta_{\tau_{d}}$ an \emph{embedding of} $T$ \emph{into a strong real form of (or rigid 
inner twist of)} $G$ \emph{of type} $J$ (\emph{cf.} \cite{kal}*{\emph{p.} 594}).

Keeping Lemma \ref{idiso} in mind, the discussion on \cite{kal}*{\emph{pp.} 593-594} reveals that the 
map
\begin{equation}
\label{kalbij}
\delta(\tau_{d}) \rightarrow \eta_{\tau_{d}}
\end{equation}
passes to a bijection from the set of equivalence classes of strong real forms of $T^{\Gamma}$ of 
type $J$ to set of equivalence classes of embeddings of $T$ into strong real forms of $G$ of type 
$J$.  Here, the two embeddings $\eta_{\tau_{1}}$, $\eta_{\tau_{2}}$ are equivalent if there exists $g \in 
G$ such that $\mathrm{Int}(g)$ is defined over $\mathbb{R}$ and
$z_{\tau_{2}}(\sigma) = g \,z_{\tau_{1}}(\sigma) \, \delta_{q} (g^{-1})$ (\emph{cf.} \cite{kal}*{\emph{p.} 591}).  

We conclude by making some cosmetic changes to the correspondences of (\ref{abvcorr}). In our current context the correspondences may be written as
\begin{equation}
\label{abvcorr1}
(\phi, \tau) \mapsto (\phi_{d}, \tau_{d}) \mapsto (\pi_{d}, \eta_{\tau_{d}}) \mapsto (\pi(\eta_{\tau_{d}}), \delta(\tau_{d})). 
\end{equation}
In this notation the L-packet (\ref{JLpacket}) becomes
\begin{equation}
\label{JLpacket1}
\Pi^{\mathrm{ABV}}_{\phi,J} = \left\{ \left( \pi(\eta_{\tau_{d}}), \delta(\tau_{d}) \right) : \tau  \in ({^\vee}G_{\phi}^{\hat{J}}/ ({^\vee}G_{\phi}^{\hat{J}})^{0})^{\wedge} 
\right\}.
\end{equation}
Alternatively, we may highlight the torus ${^d}T$ and write
\begin{equation}
\label{JLpacket2}
\Pi^{\mathrm{ABV}}_{\phi,J} = \left\{ \left( \pi(\eta_{\tau_{d}}), \delta(\tau_{d}) \right) : \tau_{d}  \in
\left( ({^d}T^{{^\vee}\theta})^{\hat{J}} /({^d}T^{{^\vee}\theta})^{\hat{J}})^{0} \langle {^\vee}\alpha(-1) : 
\alpha \in R_{\mathbb{R}}({^d}L,{^d}T) \rangle \right)^{\wedge}  
\right\}.
\end{equation}
Here, $({^d}T^{^{\vee}\theta})^{\hat{J}}$ is the preimage of
${^d}T^{^{\vee}\theta}$ under (\ref{Jhatcover}).  The identification
of the character groups follows from (\ref{compiso}) and
\cite{ABV}*{Definition 10.10}.

\section{Tempered L-packets of Kaletha, Langlands and Shelstad}
\label{klslp}

The scope of this section is narrower than the previous section.  We begin with a \emph{tempered} 
L-homomorphism $\phi : W_{\mathbb{R}} \rightarrow {^\vee}G^{\Gamma}$ rather than an arbitrary one.  Moreover, we fix $J 
\subset Z(G)^{\mathrm{tor}}$ to be a \emph{finite} subgroup defined over $\mathbb{R}$. Otherwise, 
we continue with the same setup as Section \ref{abvlp}.  Our goal is to review the parameterization 
of the tempered L-packet $\Pi_{\phi,J}^{\mathrm{KLS}}$  attached to $\phi$ and $J$ as given in 
\cite{kal}*{Section 5.6}.  The presentation in \cite{kal}*{Section 5.6} itself contains a review of the work of Shelstad in \cite{She82}.  We shall only present the results of Shelstad and Kaletha necessary for the comparison with Section \ref{abvlp} later on.

We begin with an L-homomorphism as in (\ref{lhom}) whose image in bounded.  Let $y = \exp(\uppi i \lambda) \phi(j)$ as before, and define ${^\vee}\theta = \mathrm{Int}(y)$.  Attached to $\phi$ is a root subsystem
$${^\vee}\Delta_{\phi} \subset R({^\vee}G, \mathcal{T}).$$
When ${^\vee}\Delta_{\phi}$ is non-empty, it is of Dynkin type $A_{1} \times \cdots \times A_{1}$.  The automorphism ${^\vee}\theta$ fixes every 
root in $^{\vee}\Delta_{\phi}$ and negates any root vector $X_{\alpha}$ in the root space of $
\alpha \in {^\vee}\Delta_{\phi}$.  Thinking of $^{\vee}\theta$ as a Cartan involution of ${^\vee}
G$, we say that every root in ${^\vee}\Delta_{\phi}$ is imaginary and noncompact (\cite{abv}*{(12.5), 
(12.7)}).  In addition, all roots $\alpha \in {^\vee}\Delta_{\phi}$ satisfy $\langle \lambda, {^\vee}\alpha \rangle  =0$, that is to say all roots are $\phi$\emph{-singular} (\cite{abv}*{Definition 12.4}).

Shelstad determines a new maximal torus of ${^\vee}G$, essentially by
taking Cayley transforms of $\mathcal{T}$ with respect to the roots of
${^\vee}\Delta_{\phi}$ (\cite{beyond}*{Section VI.7}).  More
precisely, for specific choices of root vectors $X_{\pm \alpha}$, Shelstad defines elements
\begin{equation}
\label{cayley}
c_{\alpha} = \exp\left(\frac{\uppi i}{4} (X_{\alpha} + X_{-\alpha}) \right), \quad \alpha \in {^\vee}\Delta_{\phi}.
\end{equation}
For each $\alpha \in {^\vee}\Delta_{\phi}$ the element $w_{\alpha} = (c_{\alpha})^{2}$
 is a representative of the simple reflection of $\alpha$ in the Weyl group.  By the properties of ${^\vee}\Delta_{\phi}$ listed above, these elements are seen 
to commute with one another, $\mathrm{Ad}(c_{\alpha})$ is seen to act trivially on $\lambda$, 
and ${^\vee}\theta(c_{\alpha}) = c_{\alpha}^{-1}$.  Choose a positive system ${^\vee}
\Delta_{\phi}^{+}$ and set 
\begin{equation}
\label{Tone}
c = \prod_{\alpha \in {^\vee}\Delta_{\phi}^{+}}c_{\alpha}, \quad \mathcal{T}_{1} = c 
\mathcal{T} c^{-1}.
\end{equation}
We call the various actions induced by conjugation with $c$ the \emph{Cayley transform with 
respect to} ${^\vee}\Delta_{\phi}$.
The roots in $c \, {^\vee}\Delta_{\phi} \subset R({^\vee}G,\mathcal{T}_{1})$  are of the form $c \, 
\alpha = c_{\alpha} \alpha$  for $\alpha \in {^\vee}\Delta_{\phi}$, and we 
compute
$${^\vee}\theta (c \,\alpha) = c_{\alpha}^{-1} ({^\vee}\theta \, \alpha)=  
c_{\alpha}  (c_{\alpha})^{-2} \,  \alpha = - c \, \alpha.$$
In other words, the Cayley transform converts the imaginary roots of ${^\vee}\Delta_{\phi}$ into 
real roots $c\, {^\vee}\Delta_{\phi}$ in $R({^\vee}G,\mathcal{T}_{1})$ (\cite{abv}*{(12.5)}).

The $\phi$-singularity of the roots in ${^\vee}\Delta_{\phi}$ ensures that $c \cdot \lambda = \lambda$ 
and so $\phi(\mathbb{C}^{\times}) \subset \mathcal{T}_{1}$.  In addition,
$$\phi(j) \cdot \mathcal{T}_{1} = {^\vee}\theta (c \cdot \mathcal{T}) = c^{-1} \cdot {^\vee}\theta (\mathcal{T}) = c^{-1} \cdot \mathcal{T} = c \cdot c^{-2} \cdot \mathcal{T} = \mathcal{T}_{1}.$$
This proves that $\phi$ takes values in the group $\mathcal{T}_{1}^{\Gamma}$ generated by $\mathcal{T}_{1}$ and $\phi(j)$ (or $y$).  We define
\begin{equation}
  \label{phi1}
  \phi_{1}: W_{\mathbb{R}} \rightarrow \mathcal{T}_{1}^{\Gamma}
\end{equation}
to be the L-homomorphism given by changing the codomain of $\phi$ to
$\mathcal{T}_{1}^{\Gamma}$.  We could define $\phi_{1}$ equally well
by conjugating $\phi$ by $c$ as in \cite{kal}*{Section 5.6}, and
leaving $\mathcal{T}$ unchanged.  However, for the purposes of
comparison later on, we prefer to record the changes effected by the
Cayley transform in the maximal tori instead. 

The representations in the the L-packet of $G$ are obtained from
representations in the L-packet of an intermediate Levi subgroup $M_{1}$.
The dual group ${^\vee}M_{1}$ of $M_{1}$ is defined by the root system
\begin{align}
\label{maxsplit}
R({^\vee}M_{1}, \mathcal{T}_{1}) &= \{ \alpha \in R({^\vee}G, \mathcal{T}_{1}):  {^\vee}\theta 
(\alpha) = - \alpha\} \\
\nonumber &= \{c \,\alpha: {^\vee} \alpha \in R({^\vee}G, \mathcal{T}), \ \phi(j) \cdot \alpha = - c^{2} \, \alpha \}.
\end{align}
The group ${^\vee}M_{1}$ is a Levi subgroup of ${^\vee}G$ which
corresponds to a Levi subgroup $M_{1}$ of 
$G$ using the Borel pairs $T\subset B$ and $\mathcal{T}_{1} \subset c\mathcal{B} c^{-1}$ 
(\cite{borel}*{Section 3}).  Furthermore, $M_{1}$ has a real structure defined by $\delta_{q}$. 
Since $M_{1}$ is preserved by $\delta_{q}$, we may define 
$$M_{1}^{\Gamma}  = M_{1} \rtimes \langle \delta_{q} \rangle \subset G^{\Gamma}.$$

We continue in making our way to a description  of the intermediate L-packet $\Pi_{\phi_{1},J}^{M_{1}}$.  As $\lambda$ is not necessarily $M_{1}$-regular, we make a detour by shifting to a regular infinitesimal character, describing the shifted L-packet, and then recovering $\Pi_{\phi_{1},J}^{M_{1}}$ from the shifted L-packet.
Let $\nu \in X_{*}(\mathcal{T}_{1})$ be a strictly dominant element with respect to $c \mathcal{B} 
c^{-1} \cap {^\vee}M_{1}$ and let $\phi_{\nu}: W_{\mathbb{R}} \rightarrow \mathcal{T}_{1}^{\Gamma}$ be the L-homomorphism obtained from $\phi_{1}$ by replacing $\lambda$ with $\lambda+ \nu$. 
(\emph{cf.}  (\ref{lhom})). Then $\phi_{\nu}$ is  a discrete L-parameter for $M_{1}$.  As in 
the previous section, \cite{abv}*{Proposition 13.10} pairs $\mathcal{T}_{1}^{\Gamma}$ with a Cartan 
subgroup $T_{1}^{\Gamma} \subset M_{1}^{\Gamma}$.  Again, as in the previous section, this 
inclusion into $M_{1}^{\Gamma}$ yields an $\mathbb{R}$-embedding
\begin{equation}
  \label{Wemb}
  \eta_{\mathcal{W}}: T_{1}  \rightarrow M_{1}
\end{equation}
where the $\mathbb{R}$-structure of $T_{1}$ and $M_{1}$ are given by $\delta_{q}$.
Moreover, this embedding together with the L-homomorphism $\phi_{\nu}$ 
determines (the equivalence class of) a representation $\pi^{M_{1}}_{\nu}(\eta_{\mathcal{W}})$ which 
is generic with respect to the Whittaker datum $\mathcal{W}_{M_{1}}$ inherited from $\mathcal{W}$ 
(\cite{abv}*{Lemma 14.11}).   The representation
$\pi^{M_{1}}_{\nu}(\eta_{\mathcal{W}})$ is in the essential discrete
series of $M_{1}(\mathbb{R}, \delta_{q})$, and  the pair
$(\pi^{M_{1}}_{\nu}(\eta_{\mathcal{W}}), \delta_{q})$ is a
representation of a strong real form of type $J$ (in fact it is pure).  

Recall from (\ref{kalbij}) that there is a bijection between the set of equivalence classes of embeddings of $T_{1}$ into strong real forms of $M_{1}$ of type $J$ and  the set of equivalence classes of  strong real forms of $T_{1}^{\Gamma}$ of type $J$.  Combining this bijection with bijection (\ref{Jforms}), we recover the L-packet
\begin{equation}
\label{M1packet}
\Pi^{M_{1}}_{\phi_{\nu},J} = \left\{ \left(
\pi^{M_{1}}_{\nu}(\eta_{\tau_{1}}), \delta(\tau_{1}) \right) : \tau
\in (({^\vee}M_{1})_{\phi_{\nu}}^{\hat{J}}/
(({^\vee}M_{1})_{\phi_{\nu}}^{\hat{J}})^{0} )^{\wedge} 
\right\}.
\end{equation}
This is just the L-packet (\ref{JLpacket1}) for the group $M_{1}$ and the discrete L-parameter
 $\phi_{\nu}$.  Nevertheless, some explanation of the terms is in order.   The group $({^\vee}
 M_{1})_{\phi_{\nu}}$ is the centralizer in ${^\vee}M_{1}$ of $\phi_{\nu}(W_{\mathbb{R}})$, and 
 $({^\vee}M_{1})_{\phi_{\nu}}^{\hat{J}}$ is its preimage in the covering group 
 ${^\vee}M_{1}^{\hat{J}}$.  (By Lemma \ref{hatJ}, this covering is also the dual of $\overline{M}_{1} = 
 M_{1}/J$.)  There is an isomorphism 
 between $(\mathcal{T}_{1}^{{^\vee}\theta})^{\hat{J}} /((\mathcal{T}_{1}^{{^\vee}
 \theta})^{\hat{J}})^{0}$ and $({^\vee}M_{1})_{\phi_{\nu}}^{\hat{J}}/
(({^\vee}M_{1})_{\phi_{\nu}}^{\hat{J}})^{0}$ (\emph{cf.} (\ref{compiso})).  The character
$\tau_{1}$  of the former group is identified with $\tau$ under this isomorphism.  

The L-packet $\Pi_{\phi_{1},J}^{M_{1}}$ is defined from $\Pi^{M_{1}}_{\phi_{\nu},J} $ by  translating the representations in  $\Pi^{M_{1}}_{\phi_{\nu},J} $ to the infinitesimal character $\lambda$ using Jantzen-Zuckerman translation (\cite{greenbook}*{Definition 4.5.7}).  Each translate is either zero or an  irreducible representation of a strong real form of type $J$.  By \cite{kal}*{\emph{p.} 619},  the characters of $({^\vee}M_{1})_{\phi_{\nu}}^{\hat{J}}/
(({^\vee}M_{1})_{\phi_{\nu}}^{\hat{J}})^{0}$ which correspond to
non-zero translates are precisely those characters which are trivial
on the cosets of ${^\vee}\alpha(-1)$, where $\alpha \in c\,
{^\vee}\Delta_{\phi}$ (\emph{cf.} (\ref{alphaone})).  Consequently, the L-packet for the group $M_{1}$ and the merely tempered L-parameter $\phi_{1}$ is 
\begin{equation}
\label{M1packet1}
\Pi^{M_{1}}_{\phi_{1},J} = \left\{ \left( \pi^{M_{1}}(\eta_{\tau_{1}}), \delta(\tau_{1}) \right) : \tau_{1} \in \left( (\mathcal{T}_{1}^{{^\vee}\theta})^{\hat{J}} /((\mathcal{T}_{1}^{{^\vee}\theta})^{\hat{J}})^{0} \langle 
{^\vee}\alpha(-1) : \alpha \in c \,{^\vee}\Delta_{\phi} \rangle \right)^{\wedge}  \right\}. 
\end{equation}
The representations in this set may either be thought of as the translates of the representations of (\ref{M1packet}), or equivalently as the standard tempered representations defined by $\phi_{1}$ and the embeddings $\eta_{\tau_{1}}$ (\cite{abv}*{Proposition 11.18 (a)}).

The final step in the description of $\Pi_{\phi,J}^{\mathrm{KLS}}$ is the parabolic 
induction
\begin{equation}
  \label{klsrep}
\pi(\eta_{{\tau}}) = \mathrm{ind}_{M_{1}(\mathbb{R},\delta(\tau_{1}))}^{G(\mathbb{R},\delta(\tau_{1}))}  \,
\pi^{M_{1}}(\eta_{\tau_{1}})
\end{equation}
of each (equivalence class of a) representation occurring in (\ref{M1packet1}).  Shelstad proves that 
the resulting representations are irreducible and tempered.  Observe that $\pi(\eta_{\tau_{1}})$ 
is a standard representation, so the notation conforms with that of the previous section.  At last, we 
write the packet as 
\begin{equation}
\label{klspacket}
\Pi^{\mathrm{KLS}}_{\phi,J} = \left\{ \left( \pi(\eta_{\tau_{1}}), \delta(\tau_{1}) \right) : \tau_{1} \in \left( (\mathcal{T}_{1}^{{^\vee}\theta})^{\hat{J}} /((\mathcal{T}_{1}^{{^\vee}\theta})^{\hat{J}})^{0} \langle {^\vee}\alpha(-1) : \alpha \in c \,{^\vee}\Delta_{\phi} \rangle \right)^{\wedge}  \right\}. 
\end{equation}
There is an obvious map from $(\mathcal{T}_{1}^{{^\vee}\theta})^{\hat{J}} /((\mathcal{T}_{1}^{{^\vee}\theta})^{\hat{J}})^{0}$ into ${^\vee}G_{\phi}^{\hat{J}}/ ({^\vee}G_{\phi}^{\hat{J}})^{0}$.  By \cite{kal}*{Proposition 5.9}, this map is surjective with kernel generated by $\langle {^\vee}\alpha(-1) : \alpha \in c \,{^\vee}\Delta_{\phi} \rangle$.  One may therefore also write
$$\Pi^{\mathrm{KLS}}_{\phi,J} = \left\{ \left( \pi(\eta_{\tau_{1}}), \delta(\tau_{1}) \right) : \tau \in ({^\vee}G_{\phi}^{\hat{J}}/ ({^\vee}G_{\phi}^{\hat{J}})^{0} )^{\wedge} \right\}$$
where $\tau$ and $\tau_{1}$ correspond under the isomorphism of
character groups.

\section{The comparison of tempered L-packets}

In this section we are working under the assumptions of Section
\ref{klslp}.  In particular, the L-homomorphism $\phi$ as given in
(\ref{lhom}) is tempered, and the subgroup $J \subset Z(G)$ is finite and defined
over $\mathbb{R}$.

\subsection{A comparison of component groups}
\label{compcomp}

Consider the two groups 
\begin{equation}
\label{group1}
({^d}T^{{^\vee}\theta})^{\hat{J}} /({^d}T^{{^\vee}\theta})^{\hat{J}})^{0} \langle {^\vee}\alpha(-1) : 
\alpha \in R_{\mathbb{R}}({^d}L,{^d}T) \rangle
\end{equation}
and
\begin{equation}
\label{group2}
(\mathcal{T}_{1}^{{^\vee}\theta})^{\hat{J}} /((\mathcal{T}_{1}^{{^\vee}\theta})^{\hat{J}})^{0} \langle {^\vee}\alpha(-1) : \alpha \in c \,{^\vee}\Delta_{\phi} \rangle
\end{equation}
appearing in (\ref{JLpacket2}) and  (\ref{klspacket}) respectively.  Each of these two groups is canonically isomorphic to
 ${^\vee}G_{\phi}^{\hat{J}}/ ({^\vee}G_{\phi}^{\hat{J}})^{0}$ and
therefore there is a canonical isomorphism between the two of
them. The goal of this section is to describe the latter isomorphism
without reference to ${^\vee}G_{\phi}^{\hat{J}}/ ({^\vee}G_{\phi}^{\hat{J}})^{0}$.  
 
Recall that ${^d}T$ was chosen to be a maximal torus of ${^d}L$ (\ref{dL}).
\begin{lem}
The maximal torus $\mathcal{T}_{1}$ is a subgroup of ${^d}L$.
\end{lem}
\begin{proof}
Recall that $\mathcal{T} \subset {^d}L$, $\mathcal{T}_{1} = c\mathcal{T} c^{-1}$ (\ref{Tone}),   and $\langle \alpha, \lambda \rangle = 0$ for all $\alpha \in {^\vee}\Delta_{\phi}$.  The final property and the definition of the Cayley transform imply that $\mathrm{Int}(c) \circ \lambda = \lambda$.  We must also prove that $\mathrm{Int}(y) (c)$ also commutes with $\lambda$.  This follows from $\mathrm{Int}(y) (c) = {^\vee}\theta(c) = c^{-1}$.
\end{proof}
We now see that both ${^d}T$ and $\mathcal{T}_{1}$ are maximal tori in ${^d}L$.  The torus ${^d}T$ 
is defined by the property that ${^d}\mathfrak{t}^{-{^\vee}\theta}$ is a maximal semisimple abelian 
subalgebra of ${^d}\mathfrak{l}^{-{^\vee}\theta}$. If $\mathcal{T}_{1}$ satisfies the defining property of ${^d}T$ then 
we may take $\mathcal{T}_{1} = {^d}T$ and the canonical isomorphism between (\ref{group1}) and
 (\ref{group2}) reduces to the 
identity map.  

On the other hand, if $\mathcal{T}_{1}$ does not satisfy the defining property then $R({^d}L ,\mathcal{T}_{1})$ has an imaginary noncompact root $\beta$ (\cite{abv}*{Lemma 
12.10}), and the maximal torus $c_{\beta} \mathcal{T}_{1} c_{\beta}^{-1} \subset {^d}L$ 
has a larger $(-{^\vee}\theta)$-fixed subalgebra.  Proceeding recursively, we arrive to a product $c'$ 
of Cayley transforms, and a maximal torus $c' \mathcal{T} _{1}(c')^{-1} \subset {^d}L$ which satisfies the 
defining property of ${^d}T$.  Without loss of generality, 
\begin{equation}
\label{t1td}
{^d}T = c' \mathcal{T}_{1} (c')^{-1}.
\end{equation}
The goal then is to describe the canonical isomorphism between (\ref{group1}) and (\ref{group2}) under this hypothesis.  

Maintaining the dual notation only serves to clutter the exposition here.   We shall therefore write the preliminary results in the context where $G$ is a connected complex reductive algebraic group, $\theta$ is an involutive automorphism of $G$, and $T$ is a $\theta$-stable maximal subtorus of $G$.  
\begin{lem}
\label{toralquo}
Suppose $\beta \in R(G,T)$ is imaginary and noncompact with respect to $\theta$.  Let $c_{\beta}$ 
be the Cayley transform with respect to $\beta$ (\ref{cayley}), and set $c_{\beta} \cdot T= 
c_{\beta}Tc_{\beta}^{-1}$.  Then
\begin{enumerate}[label=(\alph*)]
\item $\ker(\beta) \subset T \cap (c_{\beta} \cdot T)$ and $\ker(\beta_{|T^{\theta}}) \subset (T \cap (c_{\beta} \cdot T))^{\theta}$.

\item $\ker(\beta_{|(T^{\theta})^{0}}) = ((c_{\beta} \cdot T)^{\theta})^{0}  \ \langle {^\vee}(c_{\beta}\beta)(-1) \rangle.$

\item The previous assertions induce a  monomorphism 
$$\ker(\beta_{|T^{\theta}}) / \, \ker(\beta_{|(T^{\theta})^{0}}) \hookrightarrow (c_{\beta} \cdot T)^{\theta}/ \, ((c_{\beta} \cdot T)^{\theta})^{0} \ \langle {^\vee}(c_{\beta}\beta)(-1) \rangle.$$

\item  There is a natural isomorphism
$$T^{\theta}/ (T^{\theta})^{0} \rightarrow  \ker(\beta_{|T^{\theta}}) / \, \ker(\beta_{|(T^{\theta})^{0}})$$
and so  a monomorphism
$$T^{\theta}/ (T^{\theta})^{0} \hookrightarrow (c_{\beta} \cdot T)^{\theta}/ \, 
((c_{\beta} \cdot T)^{\theta})^{0} \ \langle {^\vee}(c_{\beta}\beta)(-1) \rangle.$$
\end{enumerate}
\end{lem}
\begin{proof}
Suppose $t \in T$ and $\beta(t)=1$.  Then $\mathrm{Ad}(t)$ fixes both $X_{\beta}$ and $X_{-\beta}$, and so $t = c_{\beta}tc_{\beta}^{-1} \in T \cap (c_{\beta} \cdot T)$.  The first assertion follows.
For (b), we note that $(T^{\theta})^{0} = (1+\theta)T$ and $\theta(\beta) = \beta$.  As a result 
\begin{align}
\nonumber \ker(\beta_{|(T^{\theta})^{0}})  &=  \{ t \theta(t): t \in T, \ \beta(t\theta(t)) = 1\}\\
\nonumber &= \{ t \theta(t): t \in T, \ \beta(t)^{2} = 1\}\\
\nonumber &= \{ t\theta(t) : t \in T, \ \beta(t) = 1\} \cup \{ t\theta(t) : t \in T, \ \beta(t) = -1\} \\
& = (1+\theta) (\ker \beta) \ \cup \ \{ t\theta(t) : t \in T, \ \beta(t) = -1\} 
\label{disunion}
\end{align}
We consider the group $(1+\theta) (\ker \beta)$ first.  Let us show
that it is connected.  Let $Y$ be in the Lie algebra
$\mathfrak{t} =  (X_{*}(T)\otimes \mathbb{C})/X_{*}(T)$ and denote the
root of the Lie algebra corresponding to $\beta$ by $d\beta$.  Suppose
$\exp(2\uppi i Y) \in \ker \beta$.  This is equivalent to
$d\beta(Y) \in \mathbb{Z}$.
We may decompose $Y$ as
$$Y = \frac{d\beta(Y)}{2} {^\vee}\beta + Y_{\perp} \in
\frac{1}{2}\mathbb{Z} \,{^\vee}\beta \oplus \ker
d\beta.$$ 
 Clearly, $\exp( \ker d\beta)$ is connected, so the component
group of $\ker \beta$ is generated by elements of the form
$$\exp\left(2 \uppi i \frac{k}{2} {^\vee}\beta \right) = \exp(\uppi i
k {^\vee}\beta) = ({^\vee}\beta(-1))^{k}, \quad k \in \mathbb{Z}.$$
The component group is thus generated by ${^\vee}\beta(-1)$.  We
recall that $\beta$ is imaginary so that $\theta \circ {^\vee}\beta =
{^\vee}\beta$.  This implies
$(1+\theta)({^\vee}\beta(-1)) = {^\vee}\beta(-1)^{2} = 1$, and we deduce
that $(1+\theta)(\ker \beta)$ is connected.

As $(1+\theta) (\ker \beta)$ is connected, it is also a torus.  The first assertion tells us that it is a subtorus of the torus $((c_{\beta} \cdot T)^{\theta})^{0}$.  The dimensions of the tori are equal as may be seen from 
\begin{align*}
\dim (1+\theta) (\ker \beta) & = \dim \ker (d\beta_{|\mathfrak{t}^{\theta}})\\
&= (\dim \mathfrak{t}^{\theta}) - 1\\
& = \dim (c_{\beta} \mathfrak{t})^{\theta}\\
&= \dim ((c_{\beta} \cdot T)^{\theta})^{0}
\end{align*}
(\cite{beyond}*{(6.65b), Proposition 6.69}).
Consequently, the first set in (\ref{disunion}) is  $((c_{\beta} \cdot T)^{\theta})^{0}$.

To compute the second set in (\ref{disunion}), suppose $t \in T$ and $\beta(t) = -1$.  Observe that ${^\vee}\beta(i) \in T$ and $\beta(\, {^\vee}\beta(i)) = i^{\langle \beta, {^\vee}\beta \rangle} = -1$.  Setting $t_{1} = t\, {^\vee}\beta(i)^{-1}$, we see that $t = t_{1} \, {^\vee}\beta(i)$, where $\beta(t_{1}) = 1$.  Moreover,
$$t \theta(t) = (t_{1} \theta(t_{1})) \ ({^\vee}\beta(i) \, \theta (^{\vee}\beta(i))).$$
We have already proven that $(t_{1} \theta(t_{1}))  \in ((c_{\beta}
\cdot T)^{\theta})^{0}$.  For the second term of the product, we
compute $({^\vee}\beta(i) \, \theta (^{\vee}\beta(i))) =
{^\vee}\beta(-1)$, and  since $\beta( {^\vee}\beta(-1)) = (-1)^{2} = 1$, we have
$${^\vee}\beta(-1) = c_{\beta} \, {^\vee}\beta(-1) c_{\beta}^{-1} = {^\vee}(c_{\beta}\beta)(-1).$$
This proves (b).  Assertion (c) follows immediately.

For the isomorphism of (d), consider 
$${^\vee}\beta( \beta(t)^{1/2}), \quad t \in T.$$ 
We compute
$$\beta \left( {^\vee}\beta(\beta(t)^{1/2}) \right)  = \beta(t)^{\langle \beta, {^\vee}\beta \rangle/2} = \beta(t).$$
In addition, as $\beta$ is imaginary, ${^\vee}\beta( \beta(t)^{1/2}) \in T^{\theta}$.  In fact,
$${^\vee}\beta( \beta(t)^{1/2})  \in {^\vee}\beta( \mathbb{C}^{\times})  \subset (T^{\theta})^{0}.$$
Suppose that $t \in T^{\theta}$ and set $t_{1} = t\, {^\vee}\beta(\beta(t)^{-1/2})$.  Then $t_{1} \in 
T^{\theta}$ and $\beta(t_{1}) = 1$.    It follows that $t(T^{\theta})^{0} = t_{1} (T^{\theta})^{0}$ and in turn that 
\begin{align*}
T^{\theta}/(T^{\theta})^{0} &\cong (\ker \beta_{|T^{\theta}}) (T^{\theta})^{0}/ (T^{\theta})^{0}\\
 &\cong \ker \beta_{|T^{\theta}} / ((T^{\theta})^{0} \cap \ker \beta_{|T^{\theta}} ) \\
&\cong \ker \beta_{|T^{\theta}} / \ker \beta_{|(T^{\theta})^{0}}.
\end{align*}
The final monomorphism is now a consequence of this isomorphism and part (c).
\end{proof}
We may apply Lemma \ref{toralquo} repeatedly.  Suppose $\beta_{1} \in R(G,T)$ and
 $\beta_{2} \in R(G, c_{\beta_{1}} \cdot T)$ satisfy the hypotheses of Lemma \ref{toralquo} for the 
 respective tori.  
 Since $c_{\beta_{1}}\beta_{1} \in R(G, c_{\beta_{1}} \cdot T)$ is real with respect to 
 $\theta$ and $\beta_{2}$ is imaginary, the two roots are orthogonal. Consequently,
$$\beta_{2}\left( {^\vee}(c_{\beta_{1}} \beta_{1}) (-1)  \right) = (-1)^{\langle \beta_{2}, 
{^\vee}(c_{\beta_{1}} \beta_{1}) \rangle} = (-1)^{0} = 1,$$
and 
$$ {^\vee}(c_{\beta_{1}} \beta_{1} ) (-1)= c_{\beta_{2}} \, {^\vee}(c_{\beta_{1}} \beta_{1} (-1))  \
c_{\beta_{2}}^{-1}  \in c_{\beta_{2}} c_{\beta_{1}}\cdot T.$$
Similar arguments imply that $c_{\beta_{2}} c_{\beta_{1}} \beta_{1} = c_{\beta_{1}} \beta_{1}$ is a 
real root in $R(G, c_{\beta_{2}} c_{\beta_{1}} \cdot T)$.  
Applying Lemma \ref{toralquo} twice, we obtain monomorphisms
\begin{align*}
T^{\theta}/(T^{\theta})^{0} &\hookrightarrow (c_{\beta_{1}} \cdot T)^{\theta}/ \, 
((c_{\beta_{1}} \cdot T)^{\theta})^{0}  \langle {^\vee}(c_{\beta_{1}}\beta_{1})(-1) \rangle \\
&\hookrightarrow (c_{\beta_{2}} c_{\beta_{1}} \cdot T)^{\theta}/ \, 
((c_{\beta_{2}} c_{\beta{1}}\cdot T)^{\theta})^{0}  \langle {^\vee}(c_{\beta_{1}}\beta_{1})(-1), {^\vee}
(c_{\beta_{2}} \beta_{2}) (-1) \rangle.
\end{align*}
Arguing inductively, we obtain the following corollary.
\begin{cor}
\label{toralquo1}
Suppose $c' \cdot T$ is the maximal torus obtained by iterated Cayley transforms $c' = c_{\beta_{m}} 
\cdots c_{\beta_{1}}$ from $T$ by imaginary noncompact roots $\beta_{1}, \ldots ,\beta_{m}$ with 
respect to $\theta$.   Then there is a natural monomorphism
$$
T^{\theta}/ (T^{\theta})^{0} \hookrightarrow (c' \cdot T)^{\theta}/ \, 
((c' \cdot T)^{\theta})^{0}\ \langle {^\vee}(c_{\beta_{j}}\beta_{j})(-1) : 1 \leq j \leq m \rangle.$$
\end{cor}
\begin{prop}
\label{toralquo2}
Suppose $c'$ in (\ref{t1td}) is the iterated composition of Cayley
transforms of imaginary noncompact roots 
$\beta_{1}, \ldots, \beta_{m}$.  Then the following inclusion diagram is commutative.
$$\xymatrix{
\frac{ \mathcal{T}_{1}^{{^\vee}\theta} }{ (\mathcal{T}_{1}^{{^\vee}\theta})^{0}  }
\ar@{^{(}->}[rr] \ar@{->>}[d]  & & 
\frac{{^d}T^{{^\vee}\theta}}{ ({^d}T^{{^\vee}\theta})^{0} \, 
 \langle {^\vee}(c_{\beta_{j}}\beta_{j})(-1) : 1 \leq j \leq m \rangle }  \ar@{->>}[d] \\
\frac{ \mathcal{T}_{1}^{{^\vee}\theta} }{ (\mathcal{T}_{1}^{{^\vee}\theta})^{0}  \, 
\langle {^\vee}\alpha(-1) : \ \alpha \in c\, {^\vee}\Delta_{\phi} \rangle}
\ar@{->}[rr]^{\cong}  \ar[dr]_{\cong}
& & 
\frac{{^d}T^{{^\vee}\theta}}{ ({^d}T^{{^\vee}\theta})^{0}  \, \langle {^\vee}\alpha(-1) : \ \alpha \in 
R_{\mathbb{R}}({^d}L, {^d}T) \rangle}  \ar[dl]^{\cong}\\
& {^\vee}G_{\phi} / ({^\vee}G_{\phi})^{0} & 
}$$
\end{prop}
\begin{proof}
The top horizontal arrow is an application of Corollary \ref{toralquo1}  to the algebraic group ${^d}L$ 
containing the tori $\mathcal{T}_{1}$ and ${^d}T$.  The left vertical
arrow is obvious.  The right vertical arrow is  clear once one recalls 
that the roots $c_{\beta_{j}}\beta_{j}$  are all real 
with respect to ${^\vee}\theta$.  The bottom two arrows were addressed
in Sections \ref{abvlp} and \ref{klslp}.  The lower horizontal arrow
is induced by the upper horizontal arrow  since the roots of  $c\,
{^\vee}\Delta_{\phi}$ are real and are not altered by the imaginary Cayley transforms.  In this way the 
diagram is commutative.  The lower horizontal arrow is an isomorphism by virtue of the lower two 
isomorphisms. 
\end{proof}
We would like a version of Proposition \ref{toralquo2} which encompasses the covering groups 
(\ref{Jhatcover}) for our finite central subgroup $J$.  An application
of Lemma \ref{hatJ} with $T$ a 
maximal torus in $G$ implies that ${^\vee}G^{\hat{J}} = {^\vee}\overline{G}$, where $\overline{G} = 
G/J$.   Let $\mathcal{T}_{1}^{\hat{J}}$ and ${^d}T^{\hat{J}}$ be the respective preimages of $
\mathcal{T}_{1}$ and ${^d}T$ under (\ref{Jhatcover}).  These preimages are maximal tori in 
${^\vee}\overline{G}$.  There is an obvious bijection between $R({^\vee}G, \mathcal{T}_{1})$ and 
$R({^\vee}\overline{G}, \mathcal{T}_{1}^{\hat{J}})$.  These kinds of bijections allow one to lift the 
Cayley transforms on maximal tori in ${^\vee}G$ to Cayley transforms in ${^\vee}\overline{G}$.  In 
particular, (\ref{t1td}) lifts to an equation 
\begin{equation}
\label{t1td1}
\mathcal{T}_{1}^{\hat{J}} = c' \cdot {^d}T^{\hat{J}}.
\end{equation}  

The following lemma is an extension of Lemma \ref{toralquo} to the setting of covers.
\begin{lem}
\label{toralquo3}
Suppose 
$$1 \rightarrow J \rightarrow G^{J} \rightarrow G \rightarrow 1$$
is a covering of $G$ in which $G^{J}$ is a connected reductive algebraic group with a finite central 
subgroup $J$.  For any subgroup $H  \subset G$ let $H^{J}$ denote its preimage in $G^{J}$.  
Suppose $\beta \in R(G,T)$ is imaginary and noncompact with respect to $\theta$ and identify 
$\beta$ with its corresponding root in $R(G^{J}, T^{T})$.  Let $c_{\beta}$ 
be the Cayley transform with respect to $\beta$ (\ref{cayley}), and set $c_{\beta} \cdot T^{J}= 
c_{\beta}T^{J}c_{\beta}^{-1}$.  Then
\begin{enumerate}[label=(\alph*)]
\item $(\ker(\beta_{|T^{\theta}}))^{J} \subset ((T \cap (c_{\beta} \cdot T))^{\theta})^{J}$.

\item $(T^{\theta})^{J} = (\ker(\beta_{|T^{\theta}}))^{J} \ {^\vee}\beta(\mathbb{C}^{\times})$

\item $(\ker(\beta_{|T^{\theta}}) )^{J} \cap ((T^{\theta})^{J})^{0} = 
(((c_{\beta} \cdot T)^{\theta})^{J})^{0}  \ \langle {^\vee}(c_{\beta}\beta)(-1) \rangle.$

\item The previous assertions induce a  monomorphism 
$$(\ker(\beta_{|T^{\theta}}))^{J} / \, (\ker(\beta_{|T^{\theta}}) )^{J} \cap ((T^{\theta})^{J})^{0}
\hookrightarrow ((c_{\beta} \cdot T)^{\theta})^{J} / \, (((c_{\beta} \cdot T)^{\theta})^{J})^{0} \ \langle 
{^\vee}(c_{\beta}\beta)(-1) \rangle.$$
\item  There is a natural isomorphism
$$(T^{\theta})^{J}/ ((T^{\theta})^{J})^{0} \rightarrow  (\ker(\beta_{|T^{\theta}}))^{J} / \, (\ker(\beta_{|T^{\theta}}) )^{J} \cap ((T^{\theta})^{J})^{0}$$
and so  a  monomorphism
$$
(T^{\theta})^{J}/ ((T^{\theta})^{J})^{0} \hookrightarrow((c_{\beta} \cdot T)^{\theta})^{J} / \, 
(((c_{\beta} \cdot T)^{\theta})^{J})^{0} \ \langle {^\vee}(c_{\beta}\beta)(-1) \rangle.$$

\end{enumerate}
\end{lem}
\begin{proof}
The first assertion follows from Lemma \ref{toralquo} (a).  For (b)
suppose $t \in (T^{\theta})^{J}$.  Let $\bar{t}$ denote the image of $t$ in $T$.  By definition, $\bar{t}$ 
is fixed by $\theta$.    Set $t_{1} = t \, {^\vee}\beta( \beta(t)^{-1/2})$.  Then $\bar{t} = \bar{t}_{1} \,{^\vee}\beta( \beta(\bar{t})^{1/2})$, and following the arguments in the proof Lemma \ref{toralquo}, we find $t_{1} \in (\ker( \beta_{|T^{\theta}} ))^{J}$.  This implies  $(T^{\theta})^{J} \subset  (\ker(\beta_{|T^{\theta}}))^{J} \ {^\vee}\beta(\mathbb{C}^{\times})$.  The reverse inclusion follows from $\beta$ being imaginary.   
By substituting (b) into the left-hand side of (c),  we obtain
$$(\ker(\beta_{|T^{\theta}}) )^{J} \cap ((T^{\theta})^{J})^{0} = (\ker(\beta_{|T^{\theta}}) )^{J} \cap \ \left( (\ker(\beta_{|T^{\theta}}) )^{J} \right)^{0} \, {^\vee}\beta(\mathbb{C}^{\times}).$$
The intersection of ${^\vee}\beta(\mathbb{C}^{\times})$ with $(\ker (\beta_{|T^{\theta}}))^{J}$ is ${^\vee}{\beta}(\pm 1)$.  Furthermore, as seen in the proof of Lemma \ref{toralquo}, ${^\vee}\beta(\pm 1) = {^\vee}(c_{\beta} \beta)(\pm 1)$.   Thus, 
$$(\ker(\beta_{|T^{\theta}}) )^{J} \cap ((T^{\theta})^{J})^{0} = 
\left( (\ker(\beta_{|T^{\theta}}) )^{J} \right)^{0} \, \langle {^\vee}(c_{\beta} \beta)(-1) \rangle.$$
Clearly, $ (\ker(\beta_{|T^{\theta}}) )^{J}  \supset 
 (\ker(\beta_{|(T^{\theta})^{0}}) )^{J}$ and the two groups have the
same dimension.  This implies that $\left( (\ker(\beta_{|T^{\theta}}) )^{J} \right)^{0} = \left( (\ker(\beta_{|(T^{\theta})^{0}}) )^{J} \right)^{0}$.  Therefore Lemma \ref{toralquo} (b) implies 
$$\left( (\ker(\beta_{|T^{\theta}}) )^{J} \right)^{0} = \left( \left(((c_{\beta} \cdot T)^{\theta})^{0}
 \langle {^\vee}(c_{\beta} \beta)(-1) \rangle \right)^{J} \right)^{0} = 
 \left( ( c_{\beta} \cdot T  )^{J} \right)^{0}.$$
 This proves (c).  
 
 Part (d) is immediate from parts (a) and (c).  Part (e) follows the proof of Lemma \ref{toralquo} (d).
\end{proof}
The following proposition ensues from Lemma \ref{toralquo2} just as Proposition \ref{toralquo1} ensues from Lemma \ref{toralquo}.
\begin{prop}
\label{toralquo4}
Suppose $c'$ in (\ref{t1td1}) is the iterated composition of imaginary noncompact roots 
$\beta_{1}, \ldots, \beta_{m}$.  Then the following inclusion diagram is commutative.
$$\xymatrix{
\frac{ (\mathcal{T}_{1}^{{^\vee}\theta})^{\hat{J}} }{ ((\mathcal{T}_{1}^{{^\vee}\theta})^{\hat{J}} )^{0}  }
\ar@{^{(}->}[rr] \ar@{->>}[d]  & & 
\frac{{(^d}T^{{^\vee}\theta})^{\hat{J} }}{ (({^d}T^{{^\vee}\theta})^{\hat{J}} )^{0}  
 \langle {^\vee}(c_{\beta_{j}}\beta_{j})(-1) : 1 \leq j \leq m \rangle }  \ar@{->>}[d] \\
\frac{ (\mathcal{T}_{1}^{{^\vee}\theta})^{\hat{J}}  }{ (( \mathcal{T}_{1}^{{^\vee}\theta})^{\hat{J}} )^{0}
   \  \langle {^\vee}\alpha(-1) : \, \alpha \in c\, {^\vee}\Delta_{\phi} \rangle}
\ar@{->}[rr]^{\cong}  \ar[dr]_{\cong}
& & 
\frac{({^d}T^{{^\vee}\theta})^{\hat{J}} }{ (({^d}T^{{^\vee}\theta})^{\hat{J}} )^{0}  \ \langle {^\vee}\alpha(-1) : \,\alpha \in R_{\mathbb{R}}(({^d}L)^{\hat{J}} , ({^d}T)^{\hat{J}} ) \rangle}  \ar[dl]^{\cong}\\ 
& ({^\vee}G_{\phi})^{\hat{J}}  / (({^\vee}G_{\phi})^{\hat{J}}) ^{0} &  }$$
\end{prop}

\subsection{A comparison of strong real forms of type $J$}
\label{compsrf}

In Sections \ref{abvlp} and \ref{klslp} the dual groups of (\ref{group1}) and
(\ref{group2}) are  placed in bijection with sets of strong real forms
of type $J$.  This requires a pairing between each of the two maximal
tori ${^d}T$ and $\mathcal{T}_{1}$.  In Section \ref{abvlp} we denoted
the torus paired with ${^d}T$ by $T \subset G$, and in Section
\ref{klslp} we denoted the torus paired with $\mathcal{T}_{1}$ by
$T_{1} \subset G$.  Equation (\ref{t1td}) relates ${^d}T$ to
$\mathcal{T}_{1}$ through Cayley transforms.  We shall establish a
parallel relationship between the paired tori $T$ and $T_{1}$.  Once
this is complete, the diagram dual to Proposition \ref{toralquo4}
together with (\ref{Jforms})  produce a map between strong real forms
of $T$ and $T_{1}$.

We continue by reviewing the manner in which $\mathcal{T}_{1}$ (or any other maximal torus of ${^\vee}G$) is paired with a maximal torus $T_{1} \subset G$.  This is presented in \cite{abv}*{Section 13-14}.   We use the constructions of Section \ref{klslp}.   Starting with the Cartan subgroup $\mathcal{T}_{1}^{\Gamma} \subset {^\vee}G^{\Gamma}$ (\emph{cf.} (\ref{phi1})), one chooses a positive system $R^{+}_{\mathbb{R}}({^\vee}G, \mathcal{T}_{1})$, for the real roots with respect to ${^\vee}\theta$, making $\lambda \in X_{*}(\mathcal{T})$ dominant.  We arbitrarily fix a positive system $R^{+}_{i \mathbb{R}}({^\vee}G, \mathcal{T}_{1})$ for the set of imaginary roots with respect to ${^\vee}\theta$.  The triple $(\mathcal{T}_{1}^{\Gamma},  R^{+}_{i \mathbb{R}}({^\vee}G, \mathcal{T}_{1}), R^{+}_{\mathbb{R}}({^\vee}G, \mathcal{T}_{1}) )$ determines a unique \emph{based Cartan subgroup} for the L-group ${^\vee}G{^\Gamma}$ (\cite{abv}*{Definition 13.7, Proposition 13.8}).   There is a pairing between this based Cartan subgroup and a \emph{based Cartan subgroup of} $(G^{\Gamma},\mathcal{W})$ (\cite{abv}*{ Proposition 13.10}).  The latter is a quadruple
\begin{equation}
  \label{bcs}
(T_{1}^{\Gamma}, \mathcal{W}(T_{1}^{\Gamma}),
  R_{i\mathbb{R}}^{+}(G,T_{1}), R_{\mathbb{R}}^{+}(G,T_{1}))
\end{equation}
in which $T_{1}^{\Gamma}$ is a Cartan subgroup of $G^{\Gamma}$
(\cite{abv}*{Definition 12.1}), $\mathcal{W}(T_{1}^{\Gamma})$ is a
$T_{1}$-conjugacy class of an element in $T_{1}^{\Gamma}-T_{1}$, and
the last two terms are positive systems for the imaginary and real
roots respectively, relative to any element in
$\mathcal{W}(T_{1}^{\Gamma})$.  This quadruple satisfies additional
properties (\cite{abv}*{Definition 13.5}), and is determined uniquely
up to conjugacy by $G$.  After possibly conjugating by an element of
$G$, we may assume that $\mathcal{W}(T_{1}^{\Gamma})$ is the
$T_{1}$-conjugacy class of the quasisplit strong real form
$\delta_{q}$ so that $T_{1}^{\Gamma} = T_{1} \rtimes  \langle
\delta_{q} \rangle$ (\ref{extG}).

The pairing between $\mathcal{T}_{1}$ and $T_{1}$ is an isomorphism
\begin{equation}
  \label{pair}
  \zeta: {^\vee}T_{1} \rightarrow \mathcal{T}_{1}
\end{equation}
which transports the Galois action on $T_{1}$ compatibly to the Galois action of $\mathcal{T}_{1}$, carries the roots of $T_{1}$ to the coroots of $\mathcal{T}_{1}$, carries imaginary roots $R_{i\mathbb{R}}^{+}(G,T_{1})$ to the positive real coroots of $\mathcal{T}_{1}$, and carries the real roots $R_{\mathbb{R}}^{+}(G,T_{1})$ to the positive imaginary coroots of $\mathcal{T}_{1}$ (\cite{abv}*{Definition 13.9}).

The pairing (\ref{pair}), together with the map (\ref{9.10map1}), yields a bijection between the strong real forms of $T_{1}$ of type $J$ and the group $\left((\mathcal{T}_{1}^{{^\vee}\theta})^{\hat{J}}/ 
((\mathcal{T}_{1}^{{^\vee}\theta})^{\hat{J}} )^{0}\right)^{\wedge}$ (\emph{cf.} (\ref{Jforms})).  By regarding the dual of (\ref{group2}) as subgroup of $\left( (\mathcal{T}_{1}^{{^\vee}\theta})^{\hat{J}}/ 
((\mathcal{T}_{1}^{{^\vee}\theta})^{\hat{J}} )^{0} \right)^{\wedge}$,
we obtain the strong real forms $\delta(\tau)$ of type $J$ appearing
in L-packet (\ref{klspacket}).
We wish to compare these strong real forms $\delta(\tau)$ with those
appearing in the L-packet  (\ref{JLpacket2}).

To facilitate this comparison, we wish  to express the maximal torus
$T$, which is paired with ${^d}T$, as an iterated Cayley transform of
$T_{1}$.   The first step in this direction is to construct a based Cartan subgroup from an iterated Cayley transform of $T_{1}$ parallel to (\ref{t1td}).  We shall achieve this one Cayley transform at a time.

Suppose $\beta \in R_{i\mathbb{R}}({^d}L, \mathcal{T}_{1})$ is
an imaginary noncompact root.  Then the pairing $\zeta$ sends $\beta$ to a
real coroot of $T_{1}$.  This coroot corresponds to a unique real
root in $R_{\mathbb{R}}(G,T_{1})$.  To keep the notation from getting
out of hand, we
also denote this real root by $\beta$.  The context should make it
clear whether $\beta \in R_{i\mathbb{R}}({^d}L, \mathcal{T}_{1})$ or
$\beta \in R_{\mathbb{R}}(G,T_{1})$.

The root $\beta \in R_{\mathbb{R}}(G,T_{1})$  is orthogonal to
$\lambda$ and satisfies a
\emph{parity condition} relative to the character of
$T_{1}(\mathbb{R}, \delta_{q})$ determined by (\ref{phi1}) (\cite{abv}*{Proposition 13.12}).  The parity
condition is the one appearing in the Hecht-Schmid character identity
theorem (\cite{abv}*{\emph{pp.} 129-130}).

Since $\beta \in R_{\mathbb{R}}(G,T_{1})$ is real, one may  apply a Cayley transform $d_{\beta} \in
G$ to obtain a maximal torus $d_{\beta} T_{1} d_{\beta}^{-1} \subset
G$ which is more compact than $T_{1}$  (\cite{beyond}*{(6.67)}).
Computations similar to those in Section \ref{klslp} reveal that
$d_{\beta} \delta_{q} d_{\beta}^{-1} \delta_{q}^{-1} = d_{\beta}^{2}$
is a representative for the simple reflection $w_{\beta}$ of $\beta$ in the Weyl group, and therefore that $\delta_{q}$ normalizes $d_{\beta} T_{1} d_{\beta}^{-1}$.  Thus, we may define 
\begin{equation}
  \label{extTone}
(d_{\beta} T_{1} d_{\beta}^{-1})^{\Gamma} = d_{\beta} T_{1}
  d_{\beta}^{-1} \rtimes \langle \delta_{q} \rangle.
\end{equation}
This is a Cartan subgroup of $G^{\Gamma}$ (\cite{abv}*{Definition 12.1}).  Let
$R^{+}_{i\mathbb{R}}(G, d_{\beta} T_{1} d_{\beta}^{-1})$ be any system of positive
imaginary roots containing $R^{+}_{i \mathbb{R}}(G,T_{1})$ and
$d_{\beta}\beta$.  Let $R^{+}_{\mathbb{R}}(G, d_{\beta}T_{1}
d_{\beta}^{-1})$ be any positive system of real roots.  Changes in
the choice of positive roots do not affect the constructions below, up
to equivalence (see proof of \cite{abv}*{Proposition 13.13}).

The parity condition and the Hecht-Schmid character identity theorem are
pertinent to the irreducible standard representations
$\pi(\eta_{\tau_{1}})$ (\ref{klsrep}) appearing  in the L-packet
(\ref{klspacket}).  The centralizer in $G(\mathbb{R}, \delta(\tau_{1}))$
of the split component of $(d_{\beta}T_{1} d_{\beta}^{-1})(\mathbb{R})$ yields
a real Levi subgroup $M_{\beta} \supsetneq M_{1}$ (\ref{maxsplit}).  The Hecht-Schmid
character identity (\cite{Speh-Vogan}*{\emph{pp.} 264-265}) and the irreducibility of $\pi(\eta_{\tau_{1}})$
provide an identity of the form 
\begin{equation}
  \label{hechtschmid}
\mathrm{ind}_{M_{1}(\mathbb{R},
  \delta(\tau_{1}))}^{M_{\beta}(\mathbb{R}, \delta(\tau_{1}))} \,
\pi^{M_{1}}(\eta_{\tau_{1}}) =  \pi^{M_{\beta}}(\eta_{\tau_{1}})
\end{equation}
The term on the left is the parabolically induced
character (equivalence class of a representation) of
$\pi^{M_{1}}(\eta_{\tau_{1}})$ (\emph{cf.} (\ref{M1packet1})).  The
term on the right is the character of an irreducible essential limit of
discrete series representation on $M_{\beta}(\mathbb{R}, \delta(\tau_{1}))$.  In addition, we have an equivalence
$$\pi(\eta_{\tau_{1}}) = \mathrm{ind}_{M_{\beta}(\mathbb{R},
  \delta(\tau_{1}))}^{G(\mathbb{R}, \delta(\tau_{1}))} \, \pi^{M_{\beta}}(\eta_{\tau_{1}})$$
of standard representations.  The difference in perspective is not entirely
apparent from our notation.  The Hecht-Schmid identity converts the
$\mathbb{R}$-embedding $\eta_{\tau_{1}}: T_{1} \rightarrow
G$ on the left into the identically denoted
$\mathbb{R}$-embedding $\eta_{\tau_{1}}: d_{\beta}T_{1}d_{\beta}^{-1} \rightarrow
G$ on the right (see \cite{knapp}*{Theorem 14.71}).  

\begin{lem}
  \label{newbcs}
  The quadruple
  $$\left((d_{\beta} T_{1} d_{\beta}^{-1})^{\Gamma},
  \mathrm{Int}(d_{\beta} T_{1}
d_{\beta}^{-1})(\delta_{q}), R^{+}_{i\mathbb{R}}(G, d_{\beta} T_{1} d_{\beta}^{-1}), R^{+}_{\mathbb{R}}(G, d_{\beta}T d_{\beta}^{-1}) \right)$$
is a based Cartan subgroup of $(G^{\Gamma},\mathcal{W})$.
\end{lem}
\begin{proof}
  We are required to verify conditions (a)-(e) of
  \cite{abv}*{Definition 13.5}.  The only conditions which are not
  obviously satisfied are (c) and (e).  Condition (c) states that
  every element in the $d_{\beta} T_{1}
d_{\beta}^{-1}$-conjugacy class of $\delta_{q}$ (\ref{extTone}) appears as the first entry in some triple
in $\mathcal{W}$ as given in Section \ref{strongrigid}.  It suffices to show that $\delta_{q}$ appears as the first entry of a triple in $\mathcal{W}$, and this is true by
(\ref{whittdatum}).

Condition  (e) states that every standard representation of $G(\mathbb{R},
\delta_{q})$ induced from a
character on $d_{\beta}T_{1} d_{\beta}^{-1}(\mathbb{R})$, dominant with
respect to $R^{+}_{i\mathbb{R}}(G, d_{\beta} T_{1} d_{\beta}^{-1})$,
has a Whittaker model with respect to $\mathcal{W}$. The discussion on
\cite{abv}*{\emph{pp.} 161-162} reduces this condition to proving that
the limit of discrete series representation
$\pi^{M_{\beta}}(\eta_{\mathcal{W}})$  has a Whittaker model with
respect to the Whittaker datum $\mathcal{W}_{M_{\beta}}$ inherited
from $\mathcal{W}$  ((\ref{Wemb}), \cite{abv}*{Lemma 14.11}). Since (\ref{bcs})
is a based Cartan subgroup for $(G,\mathcal{W})$ the same discussion
entails that $\pi^{M_{1}}(\eta_{\mathcal{W}})$ has a Whittaker model
with respect to $\mathcal{W}_{M_{1}}$.  By (\ref{hechtschmid}) we have
$\pi^{M_{\beta}}(\eta_{\mathcal{W}}) = \mathrm{ind}_{M_{1}(\mathbb{R},
  \delta(\tau))}^{M_{\beta}(\mathbb{R}, \delta(\tau))} \,
\pi^{M_{1}}(\eta_{\mathcal{W}})$, and so (\cite{abv}*{Lemma 14.11})  implies
that $\pi^{M_{\beta}}(\eta_{\mathcal{W}})$ has a Whittaker model with
respect to $\mathcal{W}_{M_{\beta}}$.
\end{proof}

We would now like to define a pairing between the based Cartan subgroup
in Lemma \ref{newbcs} and a based Cartan subgroup containing
$c_{\beta} \mathcal{T}_{1} c_{\beta}^{-1}$.  For this, we
use the pairing $\zeta$ (\ref{pair}) between the based Cartan subgroups
containing $T_{1}$ and $\mathcal{T}_{1}$. The isomorphism $\zeta$ may
be regarded as a pair of  isomorphisms, $X^{*}(T_{1}) \cong
X_{*}(\mathcal{T})$ and $X_{*}(T_{1}) \cong X^{*}(\mathcal{T}_{1})$.
Regarded in this manner we define
\begin{equation}
  \label{zetabeta}
\zeta_{\beta}: d_{\beta} T
d_{\beta}^{-1} \stackrel{\cong}{\rightarrow} c_{\beta}
\mathcal{T}_{1} c_{\beta}^{-1}
\end{equation}
through the commutative diagrams
$$\xymatrix{X^{*}(T_{1}) \ar[r]^{\zeta} & X_{*}(\mathcal{T}_{1})
  \ar[d]^{\mathrm{Int}(c_{\beta})} \\
X^{*}(d_{\beta} T_{1} d_{\beta}^{-1}) \ar[u]^{\mathrm{Int}(d_{\beta})} \ar[r]_{\zeta_{\beta}}&
X_{*}(c_{\beta} \mathcal{T}_{1} c_{\beta}^{-1})} \quad
\xymatrix{X_{*}(T_{1}) \ar[r]^{\zeta} & X^{*}(\mathcal{T}_{1})
  \ar[d]^{\mathrm{Int}(c_{\beta})} \\
X_{*}(d_{\beta} T_{1} d_{\beta}^{-1}) \ar[u]^{\mathrm{Int}(d_{\beta})} \ar[r]_{\zeta_{\beta}}&
X^{*}(c_{\beta} \mathcal{T}_{1} c_{\beta}^{-1})}$$
Recall that $\zeta$ carries the Galois action on $T_{1}$ compatibly to
the Galois action on $\mathcal{T}_{1}$.  To be more precise,
\cite{abv}*{Proposition 2.12} converts the antiholomorphic Galois
action of conjugation by $\delta_{q}$ on $T_{1}$ into a holomorphic action
$a_{T_{1}} \in \mathrm{Aut}(T_{1})$.  In the present case,
$$a_{T_{1}} (\lambda_{1}) = \overline{\lambda_{1} \circ
  \delta_{q}^{-1}}, \quad \lambda_{1} \in X^{*}(T_{1}).$$
Conjugation by
${^\vee}\delta_{q}$ on $\mathcal{T}_{1}$ is already holomorphic and is
denoted by $a_{\mathcal{T}_{1}} \in \mathrm{Aut}(\mathcal{T}_{1})$.
The compatibility condition is
$$\zeta \circ a_{T_{1}} \circ \zeta^{-1} = a_{\mathcal{T}_{1}}.$$
An implicit consequence of this compatibility condition is that $\zeta$
carries real roots in $R(G,T_{1})$ to imaginary coroots of $({^\vee}G, \mathcal{T}_{1})$, and carries imaginary roots in
$R(G,T_{1})$ to real coroots of $({^\vee}G, \mathcal{T}_{1})$.

Let $R^{+}_{i\mathbb{R}}({^\vee}G, c_{\beta}\mathcal{T}_{1} c_{\beta}^{-1})$ be the
set of roots corresponding to the imaginary coroots
$\zeta_{\beta}(R^{+}_{\mathbb{R}}(G, d_{\beta} T_{1}
d_{\beta}^{-1}))$.  Let $R^{+}_{\mathbb{R}}({^\vee}G,
c_{\beta}\mathcal{T}_{1} c_{\beta}^{-1})$ be the set of roots
corresponding to the real coroots $ \zeta_{\beta} (R^{+}_{i \mathbb{R}}(G, d_{\beta} T_{1}  d_{\beta}^{-1}))$
\begin{lem}
  \label{newpair}
  The isomorphism $\zeta_{\beta}$ carries the Galois action of
  $(d_{\beta} T_{1}d_{\beta}^{-1})^{\Gamma}$ compatibly to the Galois action on
  $(c_{\beta} \mathcal{T}_{1} c_{\beta}^{-1})^{\Gamma}$.  Moreover,
  this compatibility extends to a pairing (\cite{abv}*{Definition 13.9}) between the based
  Cartan subgroup of Lemma \ref{newbcs} and the based Cartan subgroup
  of ${^\vee}G$ determined by 
$$\left( (c_{\beta} \mathcal{T}_{1}
  c_{\beta}^{-1})^{\Gamma}, R^{+}_{i\mathbb{R}}({^\vee}G,
  c_{\beta}\mathcal{T}_{1} c_{\beta}^{-1}),  R^{+}_{\mathbb{R}}({^\vee}G,
c_{\beta}\mathcal{T}_{1} c_{\beta}^{-1}) \right).$$
\end{lem}
\begin{proof}
  We compute for any $\lambda_{1} \in X^{*}(T_{1})$ that
\begin{align*}
(a_{d_{\beta}T_{1} d_{\beta}^{-1}} (d_{\beta}\cdot \lambda_{1})) \circ
  \mathrm{Int}(d_{\beta})
  &= \overline{\lambda_{1} \circ \mathrm{Int}(d_{\beta}^{-1}) \circ
    \delta_{q}^{-1} \circ \mathrm{Int}(d_{\beta})}\\
  &= \overline{\lambda_{1} \circ   \mathrm{Int}(w_{\beta}) \circ
    \delta_{q}^{-1}}\\
  &= \overline{(w_{\beta}\cdot \lambda_{1}) \circ \delta_{q}^{-1}}\\
  &= a_{T_{1}}(w_{\beta} \cdot \lambda_{1}).
\end{align*}
A similar, and slightly easier, computation shows that
$$\mathrm{Int}(c_{\beta}^{-1}) \circ a_{c_{\beta}
  \mathcal{T}_{1} c_{\beta}^{-1}} \circ
\mathrm{Int}(c_{\beta}) = a_{\mathcal{T}_{1}} \circ \mathrm{Int}(w_{\beta})$$
It follows from these two computations and the definition of
$\zeta_{\beta}$ that 
\begin{align*}
\zeta_{\beta} \circ a_{d_{\beta}T_{1}d_{\beta}^{-1}} \circ
\zeta_{\beta}^{-1} &= \mathrm{Int}(c_{\beta}) \circ \zeta \circ
a_{T_{1}} \circ \mathrm{Int}(w_{\beta}) \circ \zeta^{-1} \circ
\mathrm{Int}(c_{\beta}^{-1})\\
&= \mathrm{Int}(c_{\beta})  \circ
a_{\mathcal{T}_{1}} \circ \mathrm{Int}(w_{\beta}) \circ
\mathrm{Int}(c_{\beta}^{-1})\\
&= a_{c_{\beta}
  \mathcal{T}_{1} c_{\beta}^{-1}}.
\end{align*}
This proves the compatibility of the Galois actions.  The extension of
the compatibility to a pairing  follows from the choice of positive
systems in the based Cartan subgroup containing $(c_{\beta}
\mathcal{T}_{1} c_{\beta}^{-1})^{\Gamma}$ (see the proof of
\cite{abv}*{Proposition 13.10 (a)}).
\end{proof}

Lemma \ref{newpair} may be applied repeatedly in the context of
(\ref{t1td}). If $c'$ in (\ref{t1td}) is the iterated composition of
imaginary noncompact roots $\beta_{1}, \ldots,
\beta_{m}$ then the lemma tells us that there is a pairing
between $d_{\beta_{m}} \cdots d_{\beta_{1}} T_{1} (d_{\beta_{m}}
\cdots d_{\beta_{1}})^{-1}$ and $c'\mathcal{T}_{1} (c')^{-1} = {^d}T$.  Let us denote $d_{\beta_{m}} \cdots d_{\beta_{1}} T_{1} (d_{\beta_{m}}
\cdots d_{\beta_{1}})^{-1}$ by $d'\cdot T_{1}$, so that $d'\cdot
T_{1}$ is paired with ${^d}T$.  This pairing and (\ref{Jforms}) allows
us to identify the strong real forms of $d'\cdot T_{1}$ of type $J$ with $\left( ({^d}T^{{^\vee}\theta})^{\hat{J}}/ 
(({^d}T^{{^\vee}\theta})^{\hat{J}} )^{0}\right)^{\wedge}$.  Similarly the pairing between $T_{1}$ and $\mathcal{T}_{1}$
allows us to identify the strong real forms of $T_{1}$ of type $J$
with $\left( (\mathcal{T}_{1}^{{^\vee}\theta})^{\hat{J}}/ 
((\mathcal{T}_{1}^{{^\vee}\theta})^{\hat{J}} )^{0} \right)^{\wedge}$.  We wish to make these identifications
in conjunction with the diagram of Proposition \ref{toralquo4}.

Given a finite abelian group $A$, we denote $\mathrm{Hom}(A,
\mathbb{C}^{\times})$ by $A^{\wedge}$.  The next corollary is an
application of $\mathrm{Hom}(\cdot,
\mathbb{C}^{\times})$ to Proposition \ref{toralquo4} and is a special
instance of Pontryagin duality.
\begin{cor}
 \label{dualdiag}
  The following diagram, given by restriction from the diagram in
  Proposition \ref{toralquo4}, is commutative
 $$\xymatrix{
\left( \frac{ (\mathcal{T}_{1}^{{^\vee}\theta})^{\hat{J}} }{
  ((\mathcal{T}_{1}^{{^\vee}\theta})^{\hat{J}} )^{0}  } \right)^{\wedge}
 & & \ar@{->>}[ll] 
\left( \frac{{(^d}T^{{^\vee}\theta})^{\hat{J} }}{
  (({^d}T^{{^\vee}\theta})^{\hat{J}} )^{0}   \langle
  {^\vee}(c_{\beta_{j}}\beta_{j})(-1) : 1 \leq j \leq m
  \rangle } \right)^{\wedge}   \\
\left( \frac{ (\mathcal{T}_{1}^{{^\vee}\theta})^{\hat{J}}  }{ (( \mathcal{T}_{1}^{{^\vee}\theta})^{\hat{J}} )^{0}
   \  \langle {^\vee}\alpha(-1) : \, \alpha \in c\,
      {^\vee}\Delta_{\phi} \rangle} \right)^{\wedge}
\ar@{<-}[rr]^{\cong}  \ar[dr]_{\cong}  \ar@{^{(}->}[u]
& & 
\left( \frac{({^d}T^{{^\vee}\theta})^{\hat{J}} }{
  (({^d}T^{{^\vee}\theta})^{\hat{J}} )^{0}  \ \langle
  {^\vee}\alpha(-1) : \,\alpha \in R_{\mathbb{R}}(({^d}L)^{\hat{J}} ,
  ({^d}T)^{\hat{J}} ) \rangle} \right)^{\wedge}  \ar[dl]^{\cong}   \ar@{^{(}->}[u]\\ 
& \left( ({^\vee}G_{\phi})^{\hat{J}}  / (({^\vee}G_{\phi})^{\hat{J}})
^{0} \right)^{\wedge} &  }$$
\end{cor}

\begin{cor}
\label{samedelta}
  Suppose $\tau_{d}$ is a character of $ \frac{({^d}T^{{^\vee}\theta})^{\hat{J}} }{
  (({^d}T^{{^\vee}\theta})^{\hat{J}} )^{0}  \ \langle
  {^\vee}\alpha(-1) : \,\alpha \in R_{\mathbb{R}}(({^d}L)^{\hat{J}} ,
  ({^d}T)^{\hat{J}} ) \rangle}$ which corresponds to the unique
character $\tau_{1}$ of $\frac{ (\mathcal{T}_{1}^{{^\vee}\theta})^{\hat{J}}  }{ (( \mathcal{T}_{1}^{{^\vee}\theta})^{\hat{J}} )^{0}
   \  \langle {^\vee}\alpha(-1) : \, \alpha \in c\,
      {^\vee}\Delta_{\phi} \rangle}$
as in Corollary \ref{dualdiag}.  Then the strong real form
$\delta(\tau_{d}) \in G^{\Gamma}- G$ is equivalent to
the strong real form $\delta(\tau_{1}) \in G^{\Gamma}- G$.
\end{cor}
\begin{proof}
First assume $J = \{1\}$ and $c' = c_{\beta}$ for a single imaginary
noncompact root $\beta \in R({^d}L, {^d}T)$ (\ref{t1td}).  This places
us in the setting of Lemma \ref{newpair}, that is
$d'\cdot T_{1}$ is paired with ${^d}T = c_{\beta}\cdot \mathcal{T}_{1}$
through $\zeta_{\beta}$ (\ref{zetabeta}), and $d' = d_{\beta}$  for
real  $\beta \in R(G,T_{1})$.   We regard $\tau_{d}$ as a character of 
$\frac{{^d}T^{{^\vee}\theta}}{ ({^d}T^{{^\vee}\theta})^{0} \, 
\langle {^\vee}(c_{\beta}\beta)(-1) \rangle}$ which is trivial on
${^\vee}\alpha(-1)$ for all $\alpha \in R_{\mathbb{R}}({^d}L ,
{^d}T )$.  According to Lemma \ref{toralquo} and Proposition \ref{toralquo2}, the character $\tau_{d}$
is  determined by its values on a set
$$\left \{  t_{1} (\mathcal{T}_{1}^{{^\vee}\theta})^{0}, \ldots,
t_{\ell} (\mathcal{T}_{1}^{{^\vee}\theta})^{0} \right\} \subset \mathcal{T}_{1}^{{^\vee}\theta}/ (\mathcal{T}_{1}^{{^\vee}\theta})^{0}$$
in which $t_{1}, \ldots , t_{\ell} \in \ker \beta$.  In addition, the character
$\tau_{1}$ is determined by the equations
\begin{equation}
\label{defeq}
  \tau_{1} (t_{j} (\mathcal{T}_{1}^{{^\vee}\theta})^{0}) = \tau_{d}
(t_{j} (\mathcal{T}_{1}^{{^\vee}\theta})^{0}), \quad 1 \leq j \leq
  \ell.
\end{equation}

The strong real form $\delta(\tau_{d})$ is defined through the maps in
(\ref{pureforms}) and the pairing $\zeta_{\beta}$.  To be more
precise, the second map in (\ref{pureforms}) is defined by
identifying $\tau_{d}$ with a character in $X^{*}({^d}T) =
X^{*}(c_{\beta}\cdot \mathcal{T}_{1})$ and sending
this character to  an element in $X_{*}(d_{\beta}\cdot T_{1})$ under
$\zeta_{\beta}$ (\ref{zetabeta}).  To be even more precise, the character $\tau_{d}$ is
sent first to $\tau_{d} \circ \mathrm{Int}(c_{\beta}) \in
X^{*}(\mathcal{T}_{1})$, then to $\lambda_{1} \in X_{*}(T_{1})$ under
$\zeta$, and finally to $\mathrm{Int}(d_{\beta}) \circ \lambda_{1} \in
X_{*}(d_{\beta}\cdot T_{1})$.  The strong real form $\delta(\tau_{d})$
is represented by the element
\begin{equation}
  \label{rep1}
\exp(\uppi i \, \mathrm{Int}(d_{\beta})
\circ \lambda_{1}) \, \delta_{q} \in G^{\Gamma}.
\end{equation}

Let us retrace some of these steps in view of the relationship
(\ref{defeq}) between $\tau_{1}$ and $\tau_{d}$.  From $t_{j} \in
\ker \beta$ for all $1 \leq j \leq \ell$, it follows in turn that
$\mathrm{Int}(c_{\beta}) t_{j} = t_{j}$, $\tau_{d} = \tau_{d} \circ \mathrm{Int}(c_{\beta})$, and $\tau_{1} = \tau_{d} \circ
\mathrm{Int}(c_{\beta}) \in X^{*}(\mathcal{T}_{1})$.  The strong real form
$\delta(\tau_{1})$ is defined through (\ref{pureforms}) and the
pairing $\zeta$, so $\delta(\tau_{1})$ is represented by the element
\begin{equation}
  \label{rep2}
  \exp(\uppi i \lambda_{1}) \, \delta_{q} \in G^{\Gamma},
\end{equation}
for $\lambda_{1} \in X_{*}(T_{1})$ as above.

Hence, in comparing $\delta(\tau_{1})$ with $\delta(\tau_{d})$, we are
reduced to comparing $\lambda_{1}$ with $\mathrm{Int}(d_{\beta}) \circ
\lambda_{1}$.  For the latter comparison recall from (\ref{pureforms})
that $\sigma \circ \lambda_{1} = - \lambda_{1}$.  Since $\beta \in
R(G,T_{1})$ is real, we have
$$\langle \beta, \lambda_{1} \rangle = \langle \beta \circ \sigma,
\sigma \circ \lambda_{1} \rangle = \langle \beta, -\lambda_{1}
\rangle$$
and $\langle \beta, \lambda_{1}  \rangle = 0$.  This orthogonality
implies $\mathrm{Int}(d_{\beta}) \circ \lambda_{1} = \lambda_{1}$ and
so the respective representatives (\ref{rep1}) and (\ref{rep2}) of
$\delta(\tau_{d})$ and $\delta(\tau_{1})$ are equal.

This proves the corollary when $J$ is trivial and $c' = c_{\beta}$. The proof for
non-trivial $J$ follows the same argument except that one must replace
Lemma \ref{toralquo} with Lemma \ref{toralquo3}, and replace
the maps of (\ref{pureforms}) with those of (\ref{strongforms}).  The
proof for arbitrary $c'$ is a proof by induction on the number of
imaginary noncompact roots occurring in its definition.  The details
are left to the reader.
\end{proof}

\subsection{The proof of Theorem \ref{mainthm}}
\label{thmsec}

We conclude by indicating how every representation in
$\Pi^{\mathrm{KLS}}_{\phi,J}$  (\ref{klspacket}) is equivalent to a unique representation
in $\Pi^{\mathrm{ABV}}_{\phi,J}$ (\ref{JLpacket2}) .  The main ideas have all been
presented.  All we have to do is recall them in the correct sequence.

We continue to work under the assumptions of Section \ref{klslp} and fix a
representation $(\pi(\eta_{\tau_{1}}), \delta(\tau_{1})) \in \Pi^{\mathrm{KLS}}_{\phi,J}$ of a strong real
form given by
$$\tau_{1} \in \left(
(\mathcal{T}_{1}^{{^\vee}\theta})^{\hat{J}}
/((\mathcal{T}_{1}^{{^\vee}\theta})^{\hat{J}})^{0} \langle
{^\vee}\alpha(-1) : \alpha \in c \,{^\vee}\Delta_{\phi} \rangle
\right)^{\wedge}.$$
We recall that $\pi(\eta_{\tau_{1}})$ is an irreducible standard representation of
$G(\mathbb{R}, \delta(\tau))$.  It is constructed from a character of
$T_{1}(\mathbb{R})$ and the embedding
$\eta_{\tau_{1}} : T_{1}(\mathbb{R}) \rightarrow G(\mathbb{R}, \delta(\tau_{1}))$
using cohomological and
parabolic induction.  Following the discussion after Lemma
\ref{newpair}, we wish to transform these data to obtain an equivalent
representation from an embedding $(d'\cdot T_{1})(\mathbb{R})
\rightarrow G(\mathbb{R}, \delta(\tau))$.  This is achieved by a
repeated application of Hecht-Schmid's character identity as follows.
Recall $d' = d_{\beta_{m}} \cdots d_{\beta_{1}}$.  The root $\beta_{1}
\in R(G,T_{1})$ is real and satisfies the parity condition of the
Hecht-Schmid character identity.  The Hecht-Schmid character identity
theorem then provides an embedding
$(d_{\beta_{1}}T_{1}d_{\beta_{1}}^{-1})(\mathbb{R}) \rightarrow
G(\mathbb{R}, \delta(\tau_{1}))$ so that the resulting standard
representation  is equivalent to $\pi(\eta_{\tau_{1}})$.  We repeat
this process for the remaining real roots, which all satisfy 
requisite parity conditions.  In the end, we obtain the irreducible
standard representation obtained from an embedding $(d'\cdot
T_{1})(\mathbb{R}) \rightarrow G(\mathbb{R}, \delta(\tau_{1}))$.  It is
equivalent to the original representation $\pi(\eta_{\tau_{1}})$ so we
keep this notation for it.  From this perspective, the representation
of the strong real form $(\pi(\eta_{\tau_{1}}), \delta(\tau_{1}))$ is
determined by three data:
\begin{itemize}
\item the pairing between the maximal tori $d'\cdot T_{1}$ and $c'
  \cdot \mathcal{T}_{1} = {^d}T$

\item the embedding $\eta_{\tau_{1}}: d'\cdot T_{1}(\mathbb{R})
  \rightarrow G(\mathbb{R}, \delta(\tau_{1}))$ and

\item the L-homomorphism $\phi$, whose image is contained in ${^d}T^{\Gamma}$.
\end{itemize}
These three data are precisely those which determine a representation
in $\Pi^{\mathrm{ABV}}_{\phi,J}$.  Therefore $(\pi(\eta_{\tau_{1}}),
\delta(\tau_{1}))$ is equivalent to the unique
representation
$$(\pi(\eta_{{\tau_{d}}}), \delta(\tau_{d})) \in
\Pi^{\mathrm{ABV}}_{\phi,J}$$
which indexed by
$$\tau_{d} \in \left( ({^d}T^{{^\vee}\theta})^{\hat{J}} /({^d}T^{{^\vee}\theta})^{\hat{J}})^{0} \langle {^\vee}\alpha(-1) : 
\alpha \in R_{\mathbb{R}}({^d}L,{^d}T) \rangle \right)^{\wedge}$$
such that $\delta(\tau_{d}) = \delta(\tau_{1})$.  By Corollary \ref{samedelta}
the character $\tau_{d}$ is the unique character which corresponds to
$\tau_{1}$ in the diagram of Corollary \ref{dualdiag}.  Moreover, the two characters, $\tau_{1}$ and
$\tau_{d}$, map to the same character $\tau$ of $({^\vee}G_{\phi})^{\hat{J}}
/ (({^\vee}G_{\phi})^{\hat{J}})^{0}$.  We may substitute $\tau$ for
$\tau_{1}$ or $\tau_{d}$ as we have in Sections \ref{abvlp}-\ref{klslp}.
With this substitution, what we have proven is that every element 
$(\pi(\eta_{{\tau_{1}}}), \delta(\tau)) \in\Pi^{\mathrm{KLS}}_{\phi,J}$
is equivalent to  $(\pi(\eta_{{\tau_{d}}}), \delta(\tau))
\in\Pi^{\mathrm{ABV}}_{\phi,J}$, and this is Theorem \ref{mainthm}.



\end{document}